\documentclass[preprint]{amsart}

\usepackage[T1]{fontenc}
\usepackage{amsmath,amssymb} 
\usepackage{tikz} 
\usepackage{pgf} 
\usepackage{hyperref} 
\usepackage{amsthm}
\usepackage[all]{xy} 

\newtheorem{theorem}{Theorem}[section]
\newtheorem{lemma}[theorem]{Lemma}

\newtheorem{remark}[theorem]{Remark}
\newtheorem{definition}[theorem]{Definition}
\newtheorem{proposition}[theorem]{Proposition}
\newtheorem{example}[theorem]{Example}

\begin{document}

\title
{Structure of the attractor for a non-local Chafee-Infante problem}



\author[E. M. Moreira]{Estefani M. Moreira}
\email{estefani@usp.br}

\author[J. Valero]{Jos\'e Valero}
\email{jvalero@umh.es}

\address[1]{Departamento de Matem\'{a}tica, Instituto de Ci\^{e}ncias Ma\-te\-m\'{a}\-ti\-cas e de Computa\c{c}\~{a}o, Universidade de S\~{a}o Paulo, Campus de S\~{a}o Carlos, Caixa Postal 668, S\~{a}o Carlos, 13560-970, Brazil}
\address[2]{Centro de Investigaci\'on Operativa, Universidad Miguel Hern\'andez de Elche, Avenida Universidad s/n, 03202 Elche, Spain}


	\keywords{
reaction-diffusion equations, nonlocal equations, global attractors, structure of the attractor, Chafee-Infante equation}

\subjclass[2020]{Primary: 35B40, 35B41, 35K55, 35K57, 35K59; Secondary: 35P05}

\begin{abstract}
In this article, we study the structure of the global attractor for a non-local one-\-di\-men\-sio\-nal quasilinear problem. The strong relation of our problem with a non-local version of the Chafee-Infante problem allows us to describe the structure of its attractor. For that, we made use of the Conley index and the connection matrix theories in order to find geometric information such as the existence of heteroclinic connections between the equilibria. In this way, the structure of the attractor is completely described.
\end{abstract}
\maketitle

\section{Introduction}
Reaction-diffusion equations with non-local term have become in recent years
a subject of great interest due to the fact that in some real applications
the diffusion term can be perturbed by non-local effects (see e.g. \cite%
{ChipotLovat}, \cite{ChipotMolinet}, \cite{Ch-Va-Verg}). The asymptotic behavior of solutions as times goes to infinity plays an important role in
the study of non-linear partial differential equations. In the last years,
several authors have studied the existence and properties of global
attractors for such kind of equations in both the autonomous and
nonautonomous frameworks (see \cite{Ahn}, \cite{CaHeMa15}, \cite{d0}, \cite%
{CaHeMa18}, \cite{CaHeMa18B}, \cite{CLLM}). Also, non-local
equations without uniquenes (see \cite{CAHeMa17}, \cite{CM-RVexist_charac}) and with
delay \cite{X1} have been considered. The existence and properties of
equilibria have been studied in \cite{CCM-RV21C-Inf} and \cite{Estefani2}.

In general, we consider an equation of the type%
\[
u_{t}-a(l(u\left( t\right) ))\Delta u=f(t,u\left( t\right) ),\quad t>\tau ,\
x\in \Omega ,
\]%
with Dirichlet boundary conditions on the domain $\Omega \subset \mathbb{R}%
^{n}$, where $l\left( u\right) $ is some functional (not necessarily linear)
and $a\left( \text{\textperiodcentered }\right) $ is a real function. The
most common cases in the literature consider that either $l(u\left( t\right)
)=\int_{\Omega }\Phi (x)u(t,x)dx$, for some $\Phi \in L^{2}\left( \Omega
\right) $, or $l(u\left( t\right) )=\left\Vert u\left( t\right) \right\Vert
_{H_{0}^{1}\left( \Omega \right) }^{2}.$ The last case is very interesting
because a Lyapunov function can be constructed, so the system is gradient
and it is possible to characterize the global attractor in terms of the
unstable sets of the stationary points \cite{CM-RVexist_charac}. This is a first step in order to obtain the structure of the global attractor.

In this article, we describe the inner structure of the global attractor of 
\begin{equation}\label{eq:nonlocal}
	\left\{\begin{aligned}
		&u_t = a(\|u_x\|^2) u_{xx}+\lambda f(u), \ x \in (0,\pi), \ t >0, \\
		& u(t,0)=u(t,\pi)=0, \ t \geq 0,\\
		& u(0,\cdot)=u_0 \in H^1_0(0,\pi),
	\end{aligned}\right.
\end{equation}
where $\lambda>0$, $a\in C^1( \mathbb{R}^+)$ is a non-decreasing and globally Lipschitz function with $0< m \leq a(\cdot)\leq M$, for constants $m,M>0$. The function $f\in C^2(\mathbb{R})$ is odd, with
\begin{equation}\label{eq:prop_f}
	f(0)=0, f'(0)=1, sf''(s)<0, s \in \mathbb{R} 
\end{equation}
and
\begin{equation}\label{eq:dissip_f}
\limsup_{|s|\rightarrow +\infty} \frac{f(s)}{s}\leq 0.
\end{equation}

The norm $\|\cdot\|$ denotes the usual norm in $L^2(0,\pi)$.

The inner structure of a global attractor is a very challenging subject. For instance, we need global well-posedness problems for which we may identify bounded invariants sets that attracts solutions. After that, we need to establish the connections between these invariant regions. 

For the case $a(\cdot)\equiv 1$, that is,
\begin{equation}\label{eq:C-I}
	\left\{
	\begin{aligned}
		&u_t = u_{xx}+ \lambda f(u), \ x \in (0,\pi), \ t>0,\\
		&u(t,0)=u(t,\pi)=0, \ t\geq 0,\\
		&u(0,\cdot) = u_0 \in H^1_0(0,\pi).
	\end{aligned}
	\right.
\end{equation}
the structure of the attractor is well-understood. This problem, we may call it the Chafee-Infante problem, appeared in \cite{C-I-n2} and \cite{C-I74} and plays an important role in the theory of nonlinear dynamical systems. In fact, \eqref{eq:C-I} inspired a lot of works after that, by a variety of authors that contributed to it or to its variations. For example, in \cite{Phillipo}, the author studies the structure of the attractor for a quasilinear Chafee-Infante equation. 

Before we continue, we will present basic definitions and results to make clear our goals. Consider a metric space $(X,d)$ and denote by $C(X)$ the set of all the continuous functions from $X$ to itself.

\begin{definition}
	A family $\{S(t):t\geq 0\}\subset C(X)$ is called a semigroup if it satisfies the following properties:
	\begin{itemize}
		\item[i)] $S(0)=I_X$, where $I_X$ denotes the identity map of $C(X)$;
		\item[ii)] $S(t+s)=S(t)S(s)$, for all $t,s\in \mathbb{R}^+$;
		\item[iii)] The map $\mathbb{R}^+\times X \ni (t,x)\mapsto S(t)x$ is continuous.
	\end{itemize}
\end{definition}
In the case where $S(t)$ is defined for all $t \in \mathbb{R}$ and the above properties are satisfied with $\mathbb{R}^+$ replaced by $\mathbb{R}$, $\{S(t):t \in \mathbb{R}\}$ is called a \emph{group}.

In the study of the asymptotic behavior of nonlinear dynamical systems, we sometimes can show the existence of a bounded subset of $X$ which somehow gives the information about all the dynamics, the so-called global attractor.
\begin{definition}
	A global attractor for a semigroup $\{S(t):t\geq 0\}$ is a set $\mathcal{A}\subset X$ satisfying the following:
	\begin{itemize}
		\item[i)] $\mathcal{A}$ is compact;
		\item[ii)] $\mathcal{A}$ is invariant, that is, $S(t)\mathcal{A}=\mathcal{A}$, for all $t \in \mathbb{R}^+$.
		\item[iii)] $\mathcal{A}$ attracts bounded sets: For any bounded set $B\subset X$ and $\varepsilon>0$, we can find $t_0 = t_0(B,\varepsilon)\in \mathbb{R}^+$ such that, for all $t\geq t_0$,
		$$
		\sup_{x \in B}\inf_{y \in \mathcal{A}}  d(S(t)x,y)<\varepsilon.
		$$
	\end{itemize}
\end{definition}

A function $\xi:\mathbb{R}\to X$ is a \emph{global solution} of $\{S(t):t\geq 0\}$ if $S(t)\xi(s)=\xi(t+s)$, for all $t \in \mathbb{R}^+$ and $s \in \mathbb{R}$. If for some $x^* \in X$, $\xi(t)=x^*$, for all $t \in \mathbb{R}$, is a global solution, we say that $x^*$ is an equilibrium of $\{S(t):t\geq 0\}$.

The global attractor can be characterized as 
$$
\mathcal{A}=\{\xi(0): \xi: \mathbb{R}\to X \mbox{ is a global solution of } \{S(t): t\geq 0\} \}.
$$

For more details, see \cite{CLRLibro}. 

The set of equilibria of \eqref{eq:C-I} satisfies the elliptic problem
\begin{equation}\label{eq:ellip_C-I}
	\left\{
	\begin{aligned}
		&u_{xx}+ \lambda f(u)=0, \ x \in (0,\pi),\\
		&u(0)=u(\pi)=0.\\
	\end{aligned}
	\right.
\end{equation}

In the mentioned works, the authors have shown that, for every $\lambda>0$, the set of equilibria of \eqref{eq:C-I}, $\mathcal{E}_\lambda$, is finite. Moreover, as long as the parameter $\lambda>0$ increases, a bifurcation from $0$ happens, which causes the appearance of a new pair of equilibria in every bifurcation.

A semilinear problem in a neighborhood of en equilibrium can be approximated by a linear problem, which is called its linearization. See \cite[Chapter 5]{H-Book-1980}, for more details. For instance, if $\lambda>0$ and $\phi \in \mathcal{E}_\lambda$,  \eqref{eq:C-I} is approximated in a neighborhood of $\phi$ by
\begin{equation}\label{eq:C-I_lineariz}
	\left\{
	\begin{aligned}
		&u_t = u_{xx}+ \lambda f'(\phi)u, \ x \in (0,\pi), \ t>0,\\
		&u(t,0)=u(t,\pi)=0, \ t\geq 0,\\
		&u(0,\cdot) = u_0 \in H^1_0(0,\pi).
	\end{aligned}
	\right.
\end{equation}

Now, since \eqref{eq:C-I_lineariz} is an autonomous linear problem, we may study the eigenvalues of $L: D(L)\subset L^2(0,\pi)\to L^2(0,\pi)$ with $Lu=u_{xx}+\lambda f'(\phi)u$, for $u \in D(L)=H^2(0,\pi)\cap H^1_0(0,\pi)$ in order to determine the stability or instability of $\phi$ for \eqref{eq:C-I}. 

In the case where $Lu=0$, for $u \in D(L)$, implies $u=0$, we say that $\phi$ is a hyperbolic equilibrium. For any $\lambda>0$, we can also show the existence of a manifold related to each equilibrium $\phi \in \mathcal{E}_\lambda$,
where 
$$\begin{aligned}
W^u(\phi)=\{x \in X: \mbox{ there is } \eta: \mathbb{R}\to X, \mbox{ a global solution of } \eqref{eq:C-I}\mbox{ with }
\eta(0)=x \\ \mbox{ and } \|\eta(t)-\phi\|_X\stackrel{t\rightarrow -\infty}{ \longrightarrow } 0\}.
\end{aligned}
$$

$W^u(\phi)$ is called the unstable manifold of $\phi$.
We may sometimes denote $\|\eta(t)-\phi\|_X\stackrel{t\rightarrow \pm\infty}{ \longrightarrow } 0$ by $\eta(t)\stackrel{t\rightarrow \pm\infty}{ \longrightarrow } \phi$.


Inside the global attractor, we can find special bounded invariant sets, such as equilibria. 
We have the first description of the attractor of \eqref{eq:C-I}, which is the following
\begin{equation}\label{eq:UnstMan_Attractor}
	\mathcal{A}_\lambda=\bigcup_{\phi \in \mathcal{E}_\lambda} W^u(\phi).
\end{equation}

Later \cite{BF88}, \cite{agenentMS} and \cite{art-henry-MS} give us a better description of the dynamics inside the global attractor: 
\begin{itemize}
	\item[a)] For any global bounded solution $\xi:\mathbb{R}\to H^1_0(0,\pi)$, we find $\phi_j,\phi_k\in \mathcal{E}_\lambda$ such that
	\begin{equation*}
	\phi_j \stackrel{t\rightarrow -\infty}{\longleftarrow} \xi(t)  \stackrel{t\rightarrow +\infty}{\longrightarrow} \phi_k.
	\end{equation*}
	
	Also, $\phi_j$ has at least one more zero in $[0,\pi]$ than $\phi_k$, if $\xi(\cdot)$ is not an equilibrium.
	
\item[b)] If $\phi_j,\phi_k\in \mathcal{E}_\lambda$ and $\phi_j$ has at least one more zero in $[0,\pi]$ than $\phi_k$, then we find a global bounded solution $\eta:\mathbb{R}\to H^1_0(0,\pi)$  satisfying
	\begin{equation*}
	\phi_j \stackrel{t\rightarrow -\infty}{\longleftarrow} \eta(t)  \stackrel{t\rightarrow +\infty}{\longrightarrow} \phi_k.
\end{equation*}
\end{itemize}

The global attractor of \eqref{eq:C-I} has a structure that is robust under perturbations, see \cite{CLRLibro} for more details. In \cite{art-henry-MS}, the author has shown that two pairs of equilibria for \eqref{eq:C-I} intersect transversely. This was the key step to conclude that \eqref{eq:C-I} generates a Morse-Smale semigroup (see \cite{BCLBook} for more information about the Morse-Smale theory).

Later in \cite{Mischaikow95}, it was proved that problems sharing some determined properties with the Chafee-Infante problem have the same structure inside the attractor. The idea is to construct a semi conjugation that relates the dynamics inside the attractor of those problems and the dynamics of an ODE.

Comparing all the knowledge we have so far for \eqref{eq:nonlocal} with the one for the Chafee-Infante equation, we see some similar results. For instance, \eqref{eq:nonlocal} is a globally well-posed problem and the related semigroup admits a global attractor that can be described as \eqref{eq:UnstMan_Attractor}. Although we do not have the usual concept of hyperbolicity, we can say that the equilibria of \eqref{eq:nonlocal} are hyperbolic in some sense. In the next section, we will be more precise about the mentioned results.

In order to obtain the structural stability, we need to prove that all equilibria intersect transversely. This is very challenging and cannot follow the same arguments used in \cite{art-henry-MS} for \eqref{eq:C-I}.

In that situation, the arguments in \cite{Mischaikow95} fit better our purposes, since it is based on identify topological properties that are sufficient to assure the description of the attractor. 

In this paper, we will calculate the Conley index of each equilibrium of \eqref{eq:nonlocal}. For that, we will use the continuation of the Conley index. We will also show that this argument cannot be always applied if we exclude one of the hypotheses.

In section 2, we will present results that are known for \eqref{eq:nonlocal} and new ones, such as the injectivity in the attractor. In section 3, we will present the basic Conley index theory and we will calculate the Conley index of each equilibrium. In section 4, we will show that the structure of the attractor of \ref{eq:nonlocal} is the same as for the usual Chafee-Infante problem by showing that our problem satisfies the conditions required in \cite{Mischaikow95}.

\section{Known results for problem \eqref{eq:nonlocal}}

In the papers \cite{CCM-RV21C-Inf} and \cite{CLLM}, the authors have studied problem \eqref{eq:nonlocal} and constructed a sequence of bifurcations similar to the one in the Chafee-Infante equation. To be more precise:
   \begin{theorem}\label{theo:exist.equilibria}	 If $a(0)N^2< \lambda\leqslant a(0)(N+1)^2,$ then \eqref{eq:nonlocal} has exactly $2N+1$ equilibria, $\{0\}\cup\{\phi_{j}^\pm: j=1,\dots, N\}$. For $j=1,\dots, N$, 
	   	\begin{itemize}
	    \item[i)]  $\phi^+_{j}$ has $j+1$ zeros in $[0,\pi]$ and $\phi^+_{j}(\pi-x)=(-1)^{j+1} \phi^+_{j}(x)$, for all $x \in [0,\pi]$.
		\item[ii)] $\phi_{j}^-(x)=-\phi_{j}^+(x)$, for all $x \in [0,\pi]$.
   		\end{itemize}
	\end{theorem}
 We observe that in the above theorem, the set of equilibria depends on $\lambda>0$, which will not be made explicit in the the notation, just in order to simplify it.

If we assume that $f$ is not odd, we can also construct the same sequence of bifurcation described in Theorem \ref{theo:exist.equilibria}, except for the symmetries and relations between the equilibria. This result is also true for $f$ and $a$ satisfying weaker conditions, see \cite{CCM-RV21C-Inf} for more details.

It can be shown that each solution of \eqref{eq:nonlocal} up to a time-reparameterization is a solution of
\begin{equation}\label{eq:nl_semilinear}
	\left\{\begin{aligned}
		&u_t =  u_{xx}+\lambda \frac{f(u)}{a(\|u_x\|^2)}, \ x \in (0,\pi), \ t >0, \\
		& u(t,0)=u(t,\pi)=0, \ t \geq 0,\\
		& u(0,\cdot)=u_0 \in H^1_0(0,\pi).
	\end{aligned}\right..
\end{equation}
For more details, see \cite{CM-RVexist_charac} and \cite{CLLM}. In particular, the problems share the same equilibria for each parameter $\lambda>0$. In \cite{Estefani2}, the authors have shown that the equilibria of \eqref{eq:nl_semilinear} are hyperbolic, with the exception of the $0$ equilibrium for $\lambda \in \{a(0)N^2: N \in \mathbb{N}\}$. The relation between \eqref{eq:nonlocal} and \eqref{eq:nl_semilinear} can be used to define a concept of hyperbolicity for the equilibria of \eqref{eq:nonlocal}, see \cite{Estefani2} for more details.

In the above reference, it is shown that the linearization for each equilibrium $\phi$ of \eqref{eq:nl_semilinear} is given by
$$
L v=v'' +\lambda\frac{f'(\phi)}{a(\|\phi'\|^2)}v-\tfrac{2\lambda^2 a'(\|\phi'\|^2)}{a(\|\phi'\|^2)^3}f(\phi)\int_{0}^{\pi} f(\phi(s))v(s)ds.
$$

This operator can be decomposed as
\begin{equation}
	L_\varepsilon = L_0 + \varepsilon B,
\end{equation}
where $L_0v=v'' +\lambda\frac{f'(\phi)}{a(\|\phi'\|^2)}v$ is a Sturm-Liouville operator, see \cite{Sagan}, and $Bv=f(\phi)\int_{0}^{\pi} f(\phi(s))v(s)ds$ is an integral operator of rank 1 and $\varepsilon= -\tfrac{2\lambda^2 a'(\|\phi'\|^2)}{a(\|\phi'\|^2)^3}\leq 0$.

Thus $L_\varepsilon$ can be seen as a bounded perturbation of $L_0$. Such perturbation may be enough to cause loss of simplicity of the eigenvalues and the knowledge about the number of zeros of the eigenfunctions. The precious and beautiful arguments to assure transversality for the Chafee-Infante equation in \cite{art-henry-MS} are strongly based on the mentioned properties. Hence they are not applicable in our situation, at least we do not know that so far.

As a consequence of the proof of hyperbolicity in \cite{Estefani2}, we have the following result:
\begin{lemma}
Let $a(0)N^2<\lambda< a(0)(N+1)^2$ Denote by $\mathcal{E}_\lambda$ the set of equilibria of \eqref{eq:nl_semilinear} and by $dim W^u(\phi)$, the dimension of $W^u(\phi)$, $\phi \in \mathcal{E}_\lambda$. Under the same notation of Theorem \ref{theo:exist.equilibria},  $dim W^u(0)=N$ and $dimW^u(\phi_{j,\lambda}^+)=dim W^u(\phi_{j,\lambda}^-)=j-1$,  for all $j=1,\dots, N$.
\end{lemma}
 
It is important to mention that we do not have an answer to whether the equilibria are hyperbolic if $f$ is not odd. However, we still can analyze a neighborhood of each equilibrium in order to retrieve information using the Conley index theory. 

In \cite{CLLM}, the authors show the existence of a semigroup related to \eqref{eq:nonlocal} that admits a global attractor $\mathcal{A}$ in the phase space $H^1_0(0,\pi)$ (see also \cite{CM-RVexist_charac}, \cite{CCM-RV21C-Inf} for other conditions on the nonlinear function $f$, which include the possibility of non-uniqueness of solutions).

It is well known (see \cite{CM-RVexist_charac}, \cite{CCM-RV21C-Inf}, \cite{Estefani2} for more details) that the functional $%
E:H_{0}^{1}\left( 0,\pi \right) \rightarrow \mathbb{R}$ given by 
\begin{equation}
E\left( u\right) =\frac{1}{2}\int_{0}^{\left\Vert u_{x}\right\Vert
^{2}}a\left( s\right) ds-\lambda \int_{0}^{\pi }\int_{0}^{u\left( x\right)
}f\left( s\right) dsdx  \label{energy}
\end{equation}%
is a Lyapunov function.  This, together with the fact that the number of
equilibria is finite, implies that the global attractor consists of the set
of equilibria $\mathcal{E}_{\lambda }$ and the heteoclinic connections
between them, that is, if $x\in \mathcal{A}_{\lambda }$ but $x\not\in 
\mathcal{E}_{\lambda }$, then there are $z_{1},z_{2}\in \mathcal{E}_{\lambda
}$ and a global bounded solution $\xi :\mathbb{R}\rightarrow H_{0}^{1}\left(
0,\pi \right) $ such that $\xi \left( 0\right) =x$ and%
\begin{eqnarray*}
\xi \left( t\right)  &\rightarrow &z_{1}\text{ as }t\rightarrow +\infty , \\
\xi \left( t\right)  &\rightarrow &z_{2}\text{ as }t\rightarrow -\infty .
\end{eqnarray*}

The following proposition will give us the result of backward uniqueness of solutions in the attractor of \eqref{eq:nonlocal}. Although we do not have a semilinear equation, we can adapt the proof presented in \cite{Temam88}, in Chapter 3, Section 6.

\begin{proposition}\label{Prop:inject.A}
The semigroup restricted to the global attractor $\mathcal{A}$ of \eqref{eq:nonlocal} is injective. In other words, if $u:\mathbb{R}\to H^1_0(0,\pi)$ and $v:\mathbb{R}\to H^1_0(0,\pi)$ are global solutions of \eqref{eq:nonlocal} with $u(\mathbb{R})\cap v(\mathbb{R})\neq \emptyset$, then $u(\mathbb{R})=v(\mathbb{R})$.
\end{proposition}
\begin{proof}
	Suppose that we find global bounded solutions $\bar{u}$ and $\bar{v} $ with $\bar{u}(\mathbb{R})\cap \bar{v}(\mathbb{R})\neq \emptyset$.
	
	By \cite{CM-RVexist_charac} and \cite{Estefani2}, applying a change of variable, we find $u:\mathbb{R}\to H^1_0(0,\pi)$ and $v:\mathbb{R}\to H^1_0(0,\pi)$ which are solutions of \eqref{eq:nl_semilinear} related, respectively, to $\bar{u}$  and $\bar{v}$. 
	
	By construction, $u(\mathbb{R})\cap v(\mathbb{R})\neq \emptyset$. Without loss of generality, we may assume that $u(T)=v(T)$, for some $T \in \mathbb{R}$. 
	
	Define $w: \mathbb{R}\to H^1_0(0,\pi)$ as $w(t)=u(t)-v(t)$, for $t \in \mathbb{R}$. Our goal is to prove that $w(t)=0$, for all $t\in \mathbb{R}$. Suppose, by contradiction, that we can find $t_0 \in \mathbb{R}$ for which $w(t_0)\neq 0$.
	
Now, $w$ satisfies
	\begin{equation}\label{eq:difference}
		w_t =  w_{xx}+h(t),
	\end{equation}
	where $h(t)=\lambda\left[\tfrac{f(u)}{a(\|u_x\|^2)}- \tfrac{f(v)}{a(\|v_x\|^2)}\right] $. 
Observe that 
\begin{equation}\label{eq:h(t)}
	\begin{aligned}
	\|h(t)\|&\!\leq \frac{\lambda}{m^2}\!\left\| \left[a(\|v_x\|^2)\!-\!a(\|u_x\|^2)\right]\!f(u) + a(\|u_x\|^2)\!\left[f(u)-f(v)\right]\right\|\! \leq C\|w_x\|,
	\end{aligned}
\end{equation}	
for some constant $C>0$. Hence, $h(t) \in L^2(0,\pi)$, for all $t\in \mathbb{R}$. By the variation of constants formula, we can see that the problem is locally well-posed. In particular, we find 
$$
t_1 =\sup\{t \in [t_0, T]: w(s)\neq 0 \mbox{ for } s \in [t_0,t] \},
$$ 
with $t_1>t_0$ and $w(t_1)=0$.
	
	For  $t \in [t_0,t_1)$, define the functions $\Gamma(t)=\frac{\|w_x(t)\|^2}{\|w(t)\|^2}$ and	$g(t)=\log \|w(t)\|^{-1}$. Then
	$$\begin{aligned}
		\frac{1}{2}\frac{d}{dt}\Gamma(t)
		&=\frac{\left<(w_t)_x,w_x\right>}{\|w\|^2} - \frac{\|w_x\|^2}{\|w\|^4}\left<w_t,w\right>=\frac{\left<w_t,-w_{xx}-\Gamma(t)w\right>}{\|w\|^2} \\
		&=\frac{\left<w_{xx}+\Gamma(t)w,-w_{xx}-\Gamma(t)w\right>}{\|w\|^2}+\frac{\left<h-\Gamma(t)w,-w_{xx}-\Gamma(t)w\right>}{\|w\|^2}\\
		&=-\frac{\|w_{xx}+\Gamma(t)w\|^2}{\|w\|^2}+\frac{\left<h,-w_{xx}-\Gamma(t)w\right>}{\|w\|^2} \\
		&\leq -\frac{1}{2}\frac{\|w_{xx}+\Gamma(t)w\|^2}{\|w\|^2}+\frac{1}{2}\frac{\|h\|^2}{\|w\|^2}\leq C\Gamma(t).
	\end{aligned}
	$$
The last line is obtained by using the Young's inequality. Hence, for all $t \in [t_0,t_1)$, $$\Gamma(t)\leq \Gamma(t_0)+C(t-t_0).$$
	
Now, by \eqref{eq:h(t)},
	$$\begin{aligned}
		\frac{d}{dt}g(t) =-\frac{1}{2}\frac{d}{dt}\log \|w\|^{2} = -\frac{\left<w_t,w\right>}{\|w\|^{2}}=&-\frac{\left<w_{xx},w\right>}{\|w\|^{2}}-\frac{\left<h,w\right>}{\|w\|^{2}}\\
		&\leq \Gamma(t)+C\Gamma^\frac{1}{2}(t) \leq 2\Gamma(t)+ C^2.
	\end{aligned}
	$$
	
Thus, for all $t \in [t_0,t_1)$,
	$$
	\log\|w(t)\|^{-1}\leq \log \|w(t_0)\|^{-1} +[2\Gamma(t_0)+C^2](t-t_0) +C(t-t_0)^2<\infty.
	$$
	
	We have shown that $g$ is uniformly bounded in $[t_0,t_1)$, which is a contradiction with $w(t_1)=0$. The contradiction comes from assuming that we find $t_0\in \mathbb{R}$ such that $w(t_0)\neq 0$.
	
	 Therefore, $u(t)=v(t)$, for all $t \in \mathbb{R}$, so $\bar{u}(\mathbb{R})=u(\mathbb{R})=v(\mathbb{R})=\bar{v}(\mathbb{R})$, as desired.
\end{proof}

The above proposition will be considered again in the last section. 


\section{The Conley index of the equilibria}

The Conley index was introduced in \cite{Conley} in the context of semiflows on locally compact metric spaces. Later, the concept was extended to metric spaces for semiflows on not necessarily locally compact spaces by \cite{Rybakowski}. The Conley index is a concept that  gives a topological description for a neighborhood of an isolated invariant set (in our case, simply an isolated equilibrium). 

In what follows, we will calculate the Conley index for each equilibrium of \eqref{eq:nonlocal}. Before that, we present definitions and theorems that we will use here. The exposition of the subject will be brief and  focused on our applications. For more details, such as the proof of theoretical results, see \cite{Rybakowski} and \cite{Phillipo_IndCon}.

Consider a semigroup $\{S(t):t\geq 0\}\in C(X)$ and subsets $Y\subset N$ of $X$. We say that $Y$ is \emph{$N$-positively invariant} under the action of $\{S(t):t\geq 0\}$ if, for any $\tau \in \mathbb{R}^+$ and $y \in Y$, $S(t)y \in N$, for all $t \in [0,\tau]$, implies $S(t)y \in Y$, for all $t \in [0,\tau]$. Just for simplicity, denote $\cup\{S(t)y: t \in [0,\tau]\}$ by $S([0,\tau])y$.

A closed set $N\subset X$ is an \emph{isolating neighborhood} of $K$ if $K \in int(N)$ (the interior of $N$) and $K$ is closed and the largest invariant set in $N$. In that case, $K$ is called an isolated invariant set.

\begin{definition}
We say that a pair of sets $\left<N_1,N_2\right>$ is an \emph{index pair} in $N$ if satisfies:
\begin{itemize}
	\item[i)] $N_1, N_2$ are closed subsets of $N$ which are $N$-positively invariant;
	\item[ii)] $K \in int(N_1\setminus N_2);$
	\item[iii)] If for some $y \in N_1$ we find $t_0 \in \mathbb{R}^+$ such that $S(t_0)y \notin N$, then there is $\tau \in [0,t_0]$ for which $S([0,\tau])y\subset N$ and $S(\tau)y \in N_2$.
\end{itemize}
\end{definition}

Given $Y\subset X$ and $s\in \mathbb{R}^+$, define the set
$$
Y^s=\{x \in X: \mbox{ there is } y \in Y \mbox{ with } S([0,s])y \subset Y \mbox{ and } S(s)y=x\}.
$$

\begin{definition}
We say that a pair of sets $\left<N_1,N_2\right>$ is a \emph{quasi-index pair} in $N$ if it satisfies:
	\begin{itemize}
		\item[i)] There are $\tilde{N}_1 \subset X$ and $t \in \mathbb{R}^+$ such that $N_1\setminus N_2\subset \tilde{N}_1$, $\tilde{N}_1^{t}\subset N_1$ and $\left<\tilde{N}_1, N_2\right>$ is an index pair in $N$.
		\item[ii)] Either $N_1$ is $N$-positively invariant or else there is $M_1\subset X$ which is a $N\setminus N_2$-positively invariant closed subset, with $M_1\setminus N_2\subset \tilde{N}_1$ and $M_1^{s}=N_1$, for some $s\in \mathbb{R}^+$.
	\end{itemize}
\end{definition}

The existence of a quasi-index pair is assured in the case where $N$ is an \emph{admissible} set (see \cite{Rybakowski}): for all sequences $\{x_n\}_{n\in \mathbb{N}}\subset N$, $\{t_n\}_{n\in \mathbb{N}}\subset \mathbb{R}^+$, $t_n \rightarrow +\infty$, satisfying $S([0,t_n])x_n \in N$, for all $n \in \mathbb{N}$, we find a convergent subsequence of $\{S(t_n)x_n\}_{n\in \mathbb{N}}$.

 Given two closed subsets $Y,A \subset X$, we define the relation 
$$
\begin{aligned}
x \in A\cap Y &\implies x\sim y, \mbox{ for all }  y \in A\cap Y,\\
x \in Y\setminus A \mbox{ and } x\sim y &\iff y=x
\end{aligned}.
$$

So, the pointed space
$$
\left[\frac{Y}{A}, [A]\right]=\{[y]: \ y \in Y \mbox{ and } [y]=\{x\in Y: x\sim y\}\}
$$
is in fact a topological pointed space, with a topology induced by $Y$.

For an isolated invariant set $K$ that admits an admissible neighborhood $N$, we define the Conley index $I(K, S(\cdot))$ (or just $I(K)$) as the topological space given by $\left[\frac{N_1}{N_2}, [N_2]\right]$, for a quasi-index pair $\left<N_1,N_2\right>$ in $N$. The concept is well-defined, see Theorem I.9.4 in \cite{Rybakowski}. The Conley index can be calculated in situations in which the Morse index cannot. But, when both are defined, they are related, see for instance Corollary II.11.2 \cite{Rybakowski}. Another important characteristic of the Conley index is its continuation property. We will be more precise below.

Consider the family of semigroups $\{S_\tau(t):t \geq 0\}$, for $\tau \in [0,1]$. Define the set 
$$ \mathcal{S}(X) =\bigcup_{\tau \in [0,1]}\{ (K_\tau, S_\tau(\cdot)): K_\tau \mbox{ is an isolated invariant set for }S_\tau(\cdot)  \}.
$$

\begin{definition}\label{def:alpha_cont}
A function $\alpha:[0,1]\to \mathcal{S}(X)$ is called $\mathcal{S}$-continuous if, for any $\tau_0 \in [0,1]$, we find an open neighborhood $W\subset [0,1]$ of $\tau_0$ and a closed set $N\subset X$ such that: 
\begin{itemize}
\item[i)] For any $\tau\in W$, $N$ is an isolating neighborhood for $K_\tau$, where $K_\tau$ represents the largest invariant set in $N$ under the action of $\{S_\tau(t):t \geq 0\}$.
\item[ii)]
For all sequences $\{\tau_n\}_{n\in \mathbb{N}}\subset W$ with $\tau_n\rightarrow \tau_0$, $\{t_n\}_{n\in \mathbb{N}}\subset \mathbb{R}^+$, $t_n \rightarrow +\infty$, $\{x_n\}_{n\in \mathbb{N}}\subset N$ with $S_{\tau_n}([0,t_n])x_n\subset N$, $n \in \mathbb{N}$, the sequence $\{S_{\tau_n}(t_n)x_n\}_{n\in \mathbb{N}}$ has a convergent subsequence.
\item[iii)] Consider $\{\tau_n\}_{n \in \mathbb{N}},\tau_0 \in [0,1]$ with $\tau_n\rightarrow \tau$. Then, the family of semigroups $\{S_{\tau_n}(t):t\geq 0\}_{n \in \mathbb{N}}$ is continuous, that is, given  $\{t_n\}_{n \in \mathbb{N}},t_0\subset \mathbb{R}^+$ and $\{x_n\}_{n \in \mathbb{N}}\in \mathbb{N},x_0$ with 
$$
t_n\rightarrow t_0 \mbox{ and } x_n \rightarrow x_0, \mbox{ as } n\rightarrow +\infty,
$$
we have $S_{\tau_n}(t_n)x_n\rightarrow S_{\tau_0}(t_0)x_0$ as $n\rightarrow +\infty$.
\end{itemize}
\end{definition}

\begin{remark}
 In our context, the items \emph{ii)} and \emph{iii)} are satisfied if we assume for instance that $\{S_\tau(t):t \geq 0\}$, for $\tau \in [0,1]$, is collectively asymptotically compact.
\end{remark}

One important characteristic of the Conley index is its continuation property.

\begin{theorem}[Theorem I.12.2, \cite{Rybakowski}] \label{theo:ContIndex}
	Suppose that $\alpha: [0,1]\to \mathcal{S}(X)$ is $\mathcal{S}$-continuous. Then $I(K_\tau, S_\tau(\cdot))$ is constant, for all $\tau \in [0,1]$.
\end{theorem}

\begin{example}\label{exemp:ChaInfConleyIndex}
For $N^2<\lambda\leq (N+1)^2$, problem \eqref{eq:C-I} has the $2N+1$ equilibria $\{\phi_{j,\lambda}^\pm : j=1,\dots, N\}\cup\{0\}$, with $\phi_{j,\lambda}^\pm$ having $j-1$ zeros in $(0,\pi)$, $j=1,\dots, N$. See \cite{C-I-n2} and \cite{C-I74} for more details.

For each $j=1,\dots, N$, $\phi_{j,\lambda}^+$ and $\phi_{j,\lambda}^-$ are hyperbolic and 
$$dimW^u(\phi_{j,\lambda}^+) = dimW^u(\phi_{j,\lambda}^-)=j-1.$$

 By Corollary II.11.2, \cite{Rybakowski}, the indexes $I(\{\phi_{j,\lambda}^+\})$ and $I(\{\phi_{j,\lambda}^-\})$ are pointed $(j-1)$-spheres, for $j=1,\dots, N$.

On the other hand, the bifurcations occur at the zero equilibrium, so we could not expect the Conley index to be constant for all $\lambda>0$. In fact, observe that for $\lambda \in (N^2,(N+1)^2)$, $0$ is hyperbolic and $dim W^u(\{0\})=N$. Again, by Corollary II.11.2, \cite{Rybakowski}, a pointed $N$-sphere is the Conley index $I(\{0\})$, for all $\lambda \in (N^2,(N+1)^2)$.

We may also apply Theorem \ref{theo:ContIndex} to conclude that  $I(\{0\})$ is constant, for all $\lambda \in (N^2,(N+1)^2]$.
\end{example}

In what follows, we will show that we can use the information presented in this example in order to calculate the index for the equilibria of \eqref{eq:nl_semilinear}.


\begin{lemma}\label{lemma:cont.equilibria}
	Suppose that $a(0)N^2<\lambda <a(0)(N+1)^2$. Under the same notation of Theorem \ref{theo:exist.equilibria}, let $\phi \in\{\phi_{N,\lambda}^+,\phi_{N,\lambda}^-\}$ and, for $\tau \in [0,1]$, consider
	\begin{equation}\label{eq:nonlocal.tau}
		\left\{	\begin{aligned}
			&u_t = a_{\tau}(\|u_x\|^2) u_{xx}+\lambda f(u),\ x \in (0,\pi), \ t >0,\\
			& u(t,0)=u(t,\pi)=0, \ t \geq 0,\\
			& u(0,\cdot)=u_0 \in H^1_0(0,\pi),
		\end{aligned}\right.
	\end{equation}
	for $a_{\tau}(s)=a(\tau s+(1-\tau)\|\phi_x\|^2)$ and $f$ satisfying the same conditions imposed in \eqref{eq:nonlocal}.	
	
	Denote by $\mathcal{E}^\tau$ the set of equilibria of \eqref{eq:nonlocal.tau}, $\tau \in [0,1]$. Then $\mathcal{E}^\tau$ is a set with $2N+1$ elements, for all $\tau \in [0,1]$. Moreover, we have continuity of equilibria.
\end{lemma}
\begin{proof}
The function $a_\tau:\mathbb{R}\to [m,M]$ given by $a_\tau(s)=a(\tau s+(1-\tau)\|\phi_x\|^2)$ is globally Lipschitz continuous and $a_\tau$ is also a non-decreasing $C^1$-function, for each $\tau \in [0,1]$. So, these problems are well-defined and we have a semigroup $\{S_\tau(t): t\geq 0\}$ related to \eqref{eq:nonlocal.tau}, for all $\tau \in [0,1]$.


First, we will show that the cardinal number of  $\mathcal{E}^\tau$ is the same for all $\tau \in [0,1]$. Observe that $a(\|\phi_x\|^2)N^2 <\lambda$, since $\phi$ is an equilibrium for  
\begin{equation*}
	\left\{	\begin{aligned}
		&u_t = u_{xx}+\tfrac{\lambda}{a(\|\phi_x\|^2)} f(u),\ x \in (0,\pi), \ t >0,\\
		& u(t,0)=u(t,\pi)=0, \ t \geq 0,\\
		& u(0,\cdot)=u_0 \in H^1_0(0,\pi).
	\end{aligned}\right.
\end{equation*}

Using that $a$ is non-decreasing, for all $\tau \in [0,1]$, we have $$
a_{\tau}(0)N^2=a((1-\tau)\|\phi_x\|^2)N^2\leq a(\|\phi_x\|^2)N^2<\lambda.
$$
and, on the other hand, $$\lambda\leq a(0)(N+1)^2\leq a_\tau(0)(N+1)^2.$$

Therefore, for each $\tau \in [0,1]$, equation \eqref{eq:nonlocal.tau} has exactly $2N+1$ equilibria.


Suppose that $\{\tau_n\}_{n\in \mathbb{N}}\subset[0,1]$ and $\tau_n\rightarrow \tau_0$. 
Denote $$
\mathcal{E}^{\tau_n}=\left\{\phi_{j,(n)}^{+},\phi_{j,(n)}^{-}:j=0,\dots, N\right\}
$$
with $\phi_{0,(n)}^{+}=\phi_{0,(n)}^{-}=0$, $\phi_{j,(n)}^{\pm}$ have $j+1$ zeros in $[0,\pi]$ and $(\phi_{j,(n)}^{+})_x(0)>0$, for $j=1,\dots, N$.

We will prove the continuity of $\mathcal{E}^\tau$ in terms of the parameter $\tau\in [0,1]$. For that, fix a $j\in \{0,1,\dots,N\}$ and denote by $\psi^{(n)}=\phi_{j,(n)}^{+}$, $n \in \mathbb{N}$. The case $\psi^{(n)}=\phi_{j,(n)}^{-}$ follows analogously. So, for each $n \in \mathbb{N}$,
\begin{equation*}
	a_{\tau_n}(\|(\psi^{(n)})_x\|^2)(\psi^{(n)})_{xx}+\lambda f(\psi^{(n)})=0. 	
\end{equation*}

Multiplying the above equation by $\psi^{(n)}$, we find
$$
\|(\psi^{(n)})_x\|^2\leq \frac{1}{a(0)} \left<\lambda f(\psi^{(n)}),\psi^{(n)}\right> \leq C+\frac{1}{2}\|(\psi^{(n)})_x\|^2,
$$
for some $C>0$, where we have used the conditions on $f$ and Sobolev's embeddings. 

Therefore, this sequence is relatively compact in $L^2(0,\pi)$ and we may assume that 
it is convergent to some $\psi \in L^2(0,\pi)$. Since $H^1_0(0,\pi)\subset C([0,1])$, the sequence $\{\psi^{(n)}\}$ is also bounded in $C([0,1])$ so as the sequence $\{f(\psi^{(n)})\}$, by the continuity of $f$. 

Consequently,
$$
\|(\psi^{(n)})_{xx}\|\leq \tfrac{\lambda}{a(0)}\|f(\psi^{(n)})\|< +\infty
$$
and $\{\psi^{(n)}\}_{n\in\mathbb{N}}$ is bounded in $H^2(0,\pi)$. Since $H^2(0,\pi)$ is compactly embedded in $H^1_0(0,\pi)$, we may assume that the sequence is convergent to some $\tilde{\psi} \in H^1_0(0,\pi)$.

By the uniqueness of the limit and by $H^1(0,\pi)\subset L^2(0,\pi)$, it follows that $\tilde{\psi}=\psi$.

Observe that $\psi^{(n)}\rightarrow {\psi}$ in $C^1([0,\pi])$, since $H^2(0,\pi)\subset C^1([0,1])$ and such embedding is compact. 
The case $\psi^{(n)}=\phi_{j,(n)}^{-}$, $n \in \mathbb{N}$ follows similarly. For any $v \in H^1_0(0,\pi)$ and $n \in \mathbb{N}$, 
$$-a_{\tau_n}(\|(\psi^{(n)})_x\|^2)\left<(\psi^{(n)})_x, v_x\right> + \lambda\left<f(\psi^{(n)}),v\right>=0,$$
hence 
$$
-a_{\tau_0}(\|\psi_x\|^2)\left<\psi_x, v_x\right> +\lambda \left<f(\psi),v\right>=0.
$$

Now, using that $\psi \in H^2(0,\pi)$ since $\{\phi^{(n)}\}_{n \in \mathbb{N}}\rightarrow \psi$ weakly in $H^2(0,\pi)$, we conclude that 
$$
-a_{\tau_0}(\|\psi_x\|^2)\psi_{xx}+\lambda f(\psi)=0.
$$

Therefore, $\psi$ is an equilibrium of \eqref{eq:nonlocal.tau} for $\tau=\tau_0$. 

In the case $\phi^{(n)}=\phi_{0,+}^{(n)}=0$, for all $n \in \mathbb{N}$ we find $\psi=0$ and we are done. So, suppose $j\neq 0$. We only need to show that $\psi\neq 0$. For that, consider $b= \inf_{n \in \mathbb{N}}{\lambda}{[a_{\tau_n}(\|(\psi^{(n)})_x\|^2)]^{-1}}>0$. For any $\bar{\lambda} >N^2$, let $\varphi^{\bar{\lambda}}$ satisfy
$$
\varphi^{\bar{\lambda}}_{xx}+\bar{\lambda} f(\varphi^{\bar{\lambda}})=0
$$
with $\varphi^{\bar{\lambda}}$ having $j+1$ zeros in $[0,\pi]$ and $\varphi^{\bar{\lambda}}_x(0)>0$.

One important result is that the function
$$(n^2,+\infty)\ni\bar{\lambda}\mapsto \|\varphi^{\bar{\lambda}}_x\|$$ 
is increasing, \cite{CCM-RV21C-Inf}. In particular, for $r_n={\lambda}{[a_{\tau_n}(\|(\psi^{(n)})_x\|^2)]^{-1}}$, $n \in \mathbb{N}$, we have that  $0<\|\varphi_x^b\|\leq \|\varphi_x^{r_{n}}\|$. Since $\varphi_x^{r_{n}}=\psi^{(n)}$, $n \in \mathbb{N}$, and $\psi^{(n)}\rightarrow \psi$, we find  $0<\|\varphi_x^b\|\leq \|\psi_x\|$.

We conclude that $\psi = \phi_{j,+}^{\tau_0}$ by the $C^1$ convergence and we have the continuity of equilibria.
\end{proof}

\begin{lemma}\label{lemma:Cont_family}
For any $\tau \in [0,1]$, denote by $\{S_\tau(t): t\geq 0\}$ the semigroup related to \eqref{eq:nonlocal.tau}.

This family of semigroups is continuous, that is, given sequences \\ $\{\tau_n\}_{n \in \mathbb{N}},\tau_0\in [0,1]$, $\{u_0^{(n)}\}_{n\in \mathbb{N}},u_0\in H^1_0(0,\pi)$, $\{t_n\}_{n\in \mathbb{N}},t_0\in \mathbb{R}^+$ satisfying
\begin{equation}\label{eq:conv_seq}
\tau_n \rightarrow \tau_0,\ u_0^{(n)} \rightarrow u_0 \mbox{ in }H^1_0(0,\pi) \mbox{ and }
t_n \rightarrow t_0, \mbox{ as }n \rightarrow +\infty,
\end{equation}
we have $\|S_{\tau_n}(t_n)u_0^{(n)}-S_{\tau_0}(t)u\|_{H^1_0(0,\pi)}\rightarrow 0$ as $n \rightarrow +\infty$. 
\end{lemma}
\begin{proof}
For any $t>0$, denote $u(t)=S_{\tau_0}(t)$ and $u_n(t)=S_{\tau_n}(t)$, $n \in \mathbb{N}$.

Consider $A: D(A)\subset L^2(0,\pi)\to L^2(0,\pi)$, the operator 
given by $Au=u_{xx}$, for $u \in D(A)=H^2(0,\pi)\cap H^1_0(0,\pi)$, and define the problem
$$
\left\{\begin{aligned}
	&u_t = Au, \ x \in (0,\pi), \ t >0, \\
	& u(t,0)=u(t,\pi)=0, \ t \geq 0,\\
	& u(0,\cdot)=u_0 \in H^1_0(0,\pi).
\end{aligned}\right.
$$
Since $-A$ is sectorial, see \cite{CLRLibro}, the above problem defines a semigroup $\{e^{At}: t\geq 0 \}$, for $X=L^2(0,\pi)$. We can consider $X^\gamma$, $\gamma \in (0,1]$, as the fractional powers defined by $A$ (see Section 6.4.2, \cite{CLRLibro}, for more details). We also have the following inequalities:
\begin{equation}\label{eq:ineq_A}
\begin{aligned}
	&\|e^{At}u\|_{X^\gamma}\leq e^{-t}\|u\|_{X^\gamma}, \mbox{ for all } t\geq 0,\\
	&\|e^{At}u\|_{X^\gamma}\leq t^{-\gamma}e^{-t}\|u\|, \mbox{ for all } t\geq 0.
\end{aligned}
\end{equation}

 By the formula of variation of constants, 
$$
u_n(t)=e^{At}u_0^{(n)}+\int_0^{t}e^{A(t-s)}\frac{f(u_n(s))}{a_{\tau_n}(\|(u_n)_x(s)\|^2)}ds
$$
and
$$
u(t)=e^{At}u_0+\int_0^{t}e^{A(t-s)}\frac{f(u(s))}{a_{\tau_0}(\|u_x(s)\|^2)}ds.
$$

We may assume, without loss of generality that $\{t_n\}_{n\in \mathbb{N}}\subset [0,T]$, for some $T \in \mathbb{R}^+$ and $t_n> t_0$ for $n$ sufficiently large. We may define $$
C_f=\sup_{n=0,1,\dots}\{f(v(x)): v \in S_{\tau_n}([0,t])u_n, x \in [0,\pi]\}<+\infty.
$$

To show that this is true, we can use comparison with $g(u)=f(u)/a(0)$, $u \in \mathbb{R}$. See, for instance, Theorem 6.41 in \cite{CLRLibro}.

Now,
$$\begin{aligned}
\|u(t_n)-u(t_0)\|_{H^1_0(0,\pi)}\! &\!\leq \|(e^{A(t_n-t_0)}-I)u(t_0)\|_{H^1_0(0,\pi)}\\
&+\! \int_{t_0}^{t_n}\! \left\|e^{A(t_n-s)}\frac{f(u(s))}{a_{\tau_0}(\|u_x(s)\|^2)}\right\|_{H^1_0(0,\pi)}\hspace{-0.3 cm}ds\\
& \leq \|(e^{A(t_n-t_0)}-I)u(t_0)\|_{H^1_0(0,\pi)}\\ &+ \frac{C_f\pi^\frac{1}{2}}{m}\int_{t_0}^{t_n} e^{-(t_n-s)}(t_n-s)^{-\frac{1}{2}} ds.
\end{aligned}
$$

We can see that, assuming $|t_n-t_0|<1$, we find
$$
\begin{aligned}
\int_{t_0}^{t_n}\! e^{-(t_n-s)}(t_n-s)^{-\frac{1}{2}} du &=\!\int_{0}^{t_n-t_0}\! e^{-u}u^{-\frac{1}{2}} ds\\ &=\!2e^{-(t_n-t_0)}(t_n-t_0)^\frac{1}{2}+\!\int_{0}^{t_n-t_0}\! 2e^{-u}u^{\frac{1}{2}} ds\\
&\leq 2(t_n-t_0)^\frac{1}{2}+ 2(t_n-t_0).
\end{aligned}
$$

Hence 
\begin{equation}\label{eq:cont_u}
\|u(t_n)-u(t_0)\|_{H^1_0(0,\pi)}\leq \|(e^{A(t_n-t_0)}-I)u(t_0)\|_{H^1_0(0,\pi)}+ 4\frac{C_f\pi^\frac{1}{2}}{m}(t_n-t_0)^\frac{1}{2}.
\end{equation}

Also
$$\begin{aligned}
\left\|\frac{f(u_n(s))}{a_{\tau_n}(\|(u_n)_x(s)\|^2)}-\frac{f(u(s))}{a_{{\tau_0}}(\|u_x(s)\|^2)}\right\|
&
\leq \frac{a_{{\tau_0}}(\|u_x(s)\|^2)}{m^2} \left\|{f(u_n(s))} - {f(u(s))}\right\| \\
&\quad + \frac{I(s)}{m^2} \left\|{f(u(s))}\right\|,
\end{aligned}
$$
for $I(s) = |a_{\tau_0}(\|u_x(s)\|^2) - a_{\tau_n}(\|(u_n)_x(s)\|^2)|$.

Now, since $a$ is globally Lipschitz, there is $C_a>0$ such that
$$
\begin{aligned}
I(s) &= \left|a\!\left(\tau_0\|u_x(s)\|^2+(1-\tau_0)\|\phi_x\|^2\right) - a\!\left(\tau_n\|(u_n)_x(s)\|^2+(1-\tau_n)\|\phi_x\|^2\right)\right|\\
&\leq C_a|\tau_0\|u_x(s)\|^2 - \tau_n\|(u_n)_x(s)\|^2 +(\tau_n-\tau_0)\|\phi_x\|^2|\\
&\leq C_a|\tau_0-\tau_n|\left(\|u_x(s)\|^2 +\|\phi_x\|^2\right)+ \tau_n\|(u_n)_x(s)-u_x(s)\|^2.
\end{aligned}
$$

Thus, for some constant $C>0$,
$$
\left\|\frac{f(u_n(s))}{a_{\tau_n}(\|(u_n)_x(s)\|^2)}-\frac{f(u(s))}{a_{{\tau_0}}(\|u_x(s)\|^2)}\right\|\leq C\left[\|u_n(s)-u(s)\|_{H^1_0(0,\pi)}+|\tau_n-\tau_0|\right].
$$

Finally, using that $X^{\frac{1}{2}}=H^1_0(0,\pi)$ and \eqref{eq:ineq_A}, we have
$$
\begin{aligned}
	\|u_n\!(r)-u(r)\|_{H^1_0(0,\pi)\!}&\leq \|e^{Ar}(u_0^{(n)}-u_0)\|_{H^1_0(0,\pi)}\\ &+\!\int_0^{r}\!\left\|e^{A(r-s)}\!\left[\frac{f(u_n(s))}{a_{\tau_n}\!(\|(u_n)_x(s)\|^2)}-\frac{f(u(s))}{a_{\tau}\!(\|u_x(s)\|^2)}\right]\!\right\|_{H^1_0(0,\pi)}\!\!\!\!\!\!\!\!\!\! ds\\
	&\leq e^{-r}\|u_0^{(n)}\!-u_0\|_{H^1_0(0,\pi)}+C|\tau_n\!-\tau_0|\!\!\int_0^{r}\!\!\!\! e^{-(r-s)}\!(r-s)^{-\frac{1}{2}}\!ds\\
	&+\int_0^{r}Ce^{-(r-s)}\!(r-s)^{-\frac{1}{2}}\left\|u_n(s)-u(s)\right\|_{H^1_0(0,\pi)}ds.
\end{aligned}
$$

Taking $\psi_n(s)=e^{s}\left\|u_n(s)-u(s)\right\|_{H^1_0(0,\pi)}$, we find
$$
\psi_n(r)\leq a^n+\int_0^{r}C(r-s)^{-\frac{1}{2}}\psi_n(s)ds, 
$$
for $a^n=\|u_0^{(n)}-u_0\|_{H^1_0(0,\pi)}+ Ce^T|\tau_n-\tau_0|\Gamma(\tfrac{1}{2})$, where $\Gamma(\cdot)$ represents the Gamma function. By \cite[Lemma 6.24]{CLRLibro}, taking $K=(2C\Gamma(\tfrac{1}{2}))^2$, we have, for all $r \in [0,T]$,
$$
\psi_n(r)\leq 2a^ne^{Kr}.
$$

Hence, for all $r \in [0,T]$,
$$\left\|u_n(r)-u(r)\right\|_{H^1_0(0,\pi)}\leq 2\bar{C}e^{(K-1)r}\left(\|u_0^{(n)}-u_0\|_{H^1_0(0,\pi)} + e^T|\tau_n -\tau_0|\Gamma(\tfrac{1}{2})\right).$$

Now,
$$
\begin{aligned}
\|u_n\!(t_n)-u(t_0)\|_{H^1_0(0,\pi)}&\leq \|u(t_n)-u(t_0)\|_{H^1_0(0,\pi)}+\|u_n(t_n)-u(t_n)\|_{H^1_0(0,\pi)} \\ 
&\leq \|u(t_n)-u(t_0)\|_{H^1_0(0,\pi)}\!\!+2\bar{C}e^{(K-1)t_n}\|u_0^{(n)}\!-u_0\|_{H^1_0(0,\pi)}\\
&\quad +2\bar{C}e^{(K-1)t_n}e^T\Gamma(\tfrac{1}{2})|\tau_n -\tau_0|. 
\end{aligned}
$$

By \eqref{eq:conv_seq} and \eqref{eq:cont_u}, it follows that $\|u_n(t_n)-u(t)\|_{H^1_0(0,\pi)}\rightarrow 0$ as $n\rightarrow +\infty$.
\end{proof}


\begin{theorem}\label{theo:Conley_index}
	Suppose that $a(0)N^2<\lambda <a(0)(N+1)^2$, for $N \in \mathbb{N}$. Under the same notation of Theorem \ref{theo:exist.equilibria}, we have that the Conley index of $\phi_{j}^\pm$ is well-defined and $I(\{\phi_{j}^+\})$ and $I(\{\phi_{j}^-\})$ are pointed $(j-1)-$spheres, for $j=1,\dots,N$. Also, $I(\{0\})$ is a pointed $N-$sphere.
\end{theorem}
\begin{proof} Consider the family of semigroups presented in Lemma \ref{lemma:cont.equilibria} and fix $j=1,\dots, n$. Just to fix the notation, for each $\tau \in [0,1]$, denote by $\phi_{j,\tau}^+$ the equilibrium in $\mathcal{E}^\tau$ satisfying $\phi_{j,\tau}^+(0)=\phi_{j,\tau}^+(\tfrac{\pi}{j})=0$ and $\phi_{j,\tau}^+(x)>0$ in $(0,\tfrac{\pi}{j})$. 

Observe that $\phi_{j}^\pm=\phi_{j,1}^\pm$. We will calculate $I(\{\phi_{j}^+\})$, for  $j=1,\dots, n$. The case $\phi_{j,\tau}^-$ follows similarly.

Define $$
\begin{aligned}
	\alpha: [0,1]&\to \mathcal{S}(X)\\
	\tau&\mapsto \alpha(\tau) =[\{\phi_{j,\tau}^+\}, S_\tau(\cdot)].
\end{aligned}
$$

We want to show that $\alpha$ is $\mathcal{S}$-continuous.

For each $\tau \in [0,1]$, consider $\delta_\tau = \frac{1}{2} \inf \{ \|\psi_x -\varphi_x\|: \psi, \varphi \in \mathcal{E}^\tau,  \psi \neq \varphi\}>0$, which is well-defined since each equilibrium is isolated for $\lambda \in (a(0)N^2, a(0)(N+1)^2)$. 

Also define, for each $\tau \in [0,1]$,
$$
N_\tau = \{u \in  H^1_0(0,\pi): \|u-\phi_{j,\tau}^+\|_{H^1_0(0,\pi)}\leq\delta_\tau\}
$$
and 
$$
\begin{aligned}
V_\tau =
\{(c_\tau, d_\tau)\cap [0,1]: &\ \tau \in (c_\tau, d_\tau)\cap [0,1] \mbox{ and } N_\tau\cap \mathcal{E}^\sigma =\{\phi_{j,\sigma}^+\}\subset int(N_\tau),\\ &\mbox{ for all } \sigma \in (c_\tau, d_\tau)\cap [0,1] \}.
\end{aligned}
$$

The above sets are well-defined, by the continuity of the equilibria in the parameter $\tau$, see Lemma \ref{lemma:cont.equilibria}. Thus, $N_{\tau}$ is an isolating neighborhood of $\phi_{j,\sigma}^+$, for all $\sigma \in V_\tau$, $\tau \in [0,1]$. %
%

By Lemma \ref{lemma:Cont_family}, it only remains to show  for any $\tau_0 \in [0,1]$ that $N_{\tau_0}$ is admissible for every convergent sequence $ \{\tau_n\}_{n\in \mathbb{N}}\in V_\tau$, in the sense of Definition \ref{def:alpha_cont}, point ii).
%


Consider sequences $\{\tau_n\}_{n\in \mathbb{N}} \in [0,1]$ with $\tau_n\rightarrow \tau_0$ and $\{t_n\}_{n\in \mathbb{N}}\subset \mathbb{R}^+$ such that $t_n\rightarrow +\infty$. Suppose that, for each $n \in \mathbb{N}$, we find a solution $\psi_n:\mathbb{R}^+\to X$ of $\{S_{\tau_n}(t): t\geq 0\}$ satisfying $\psi_n([0,t_n])\subset N_{\tau_0}$. We want to prove that $\{\psi_n(t_n)\}_{n\in \mathbb{N}}$ is a convergent subsequence.


For each $n \in \mathbb{N}$, we make the change of variable in order to find a solution $w_n(\alpha_n(t))=\psi_n(t)$ of \eqref{eq:nl_semilinear}, where $\alpha_n(t)=\int_0^{t}a_{\tau_n}(\|(\psi_n)_x(r)\|^2)dr$, for $t \in [0,t_n]$. 

The function $\alpha_n$ depends on $\psi_n$ and $\alpha_n(0)=0$, thus $w_n([0,\alpha_n(t_n)])\subset N_{\tau_0}$. By the formula of variation of constants,
$$
w_n(\alpha_n(t_n))=e^{A\alpha(t_n)}\psi_n(0)+\int_0^{\alpha_n(t_n)} e^{-A(\alpha(t_n)-s)}\frac{\lambda f(w_n(s))}{a_{\tau_n}(\|(w_n)_x(s)\|^2)}ds
$$

Then, for $\gamma \in (\frac{1}{2},1]$, we get
$$\begin{aligned}
\|w_n(\alpha_n(t_n))\|_{X^\gamma}\! \leq\! \|e^{A\alpha(t_n)}\!\psi_n(0)\|_{X^\gamma}\! +\! \int_0^{\alpha_n(t_n)}\!\frac{\lambda\left\| e^{A[\alpha(t_n)-s]}f(w_n(s))\right\|_{X^\gamma}}{a_{\tau_n}(\|(w_n)_x(s)\|^2)}\! ds\\
 \leq [{\alpha(t_n)}]^{\frac{1}{2}-\gamma}\|(\psi_n)_x(0)\|\! +\! \int_0^{\alpha_n(t_n)}\! \tfrac{\lambda[\alpha(t_n)-s]^{-\gamma}}{a_{\tau_n}(\|w_x(s)\|^2)}e^{-(\alpha(t_n)-s)}\|f(w_n(s))\|ds.
\end{aligned}
$$
Using that $f$ is continuous, $m \leq a_{\tau_n}(t)$, for all $t \in \mathbb{R}$, and $H^1_0(0,\pi)\subset L^\infty(0,\pi)$, we find a constant $C>0$ such that
$$\begin{aligned}
	\|w(\alpha_n(t_n))\|_{X^\gamma} & \leq {\alpha(t_n)}^{\frac{1}{2}-\gamma}\delta +\frac{C}{m} \int_0^{\alpha_n(t_n)} e^{-(\alpha(t_n)-s)}[\alpha(t_n)-s]^{\frac{1}{2}-\gamma} ds \\
	& \leq {\alpha(t_n)}^{\frac{1}{2}-\gamma}\delta +\frac{C}{m} \int_0^{+\infty} e^{-\tau}\tau^{\frac{1}{2}-\gamma} d\tau\\
\end{aligned}
$$	

Since $\gamma-\frac{1}{2}>0$, we find $\tilde{M}>0$ such that 
\begin{equation}\label{eq:bounded_frac_space}
\|w(\alpha_n(t_n))\|_{X^\gamma}\leq \tilde{M},
\end{equation}
for all $n \in \mathbb{N}$.

Therefore, the sequence $\{w(\alpha_n(t_n))=\psi(t_n)\}_{n \in \mathbb{N}}$ is pre-compact in $H^1_0(0,\pi)$.

Now, $\psi(t_n)\in N_{\tau_0}$, for all $n \in \mathbb{N}$, and $N_{\tau_0}$ is closed, so we find a subsequence $\{\psi(t_n)\}$ which is convergent to some point in $N_{\tau_0}$.

We conclude that $\alpha$ is $S$-continuous.

The proof for $\phi_j^-$ and $0$ is similar.

Using Theorem \ref{theo:ContIndex} and Example \ref{exemp:ChaInfConleyIndex}, for $\tau=0$, we conclude that, for $\lambda \in (a(0)N^2,a(0)(N+1)^2)$, $I(\{0\})$ is a pointed $N$-sphere and $I(\{\phi_{j}^+\})$ and $I(\{\phi_{j}^-\})$ are pointed $(j-1)$-spheres, for $j=1,\dots, N$.
\end{proof}

\section{Structure of the global attractor}

For our purposes, we need to present the concept of a connection matrix for a Morse decomposition. This theory was developed by Franzosa in \cite{Franzosa86IF,Franzosa89CM}. Later, the author also developed a concept of transition matrix, see \cite{Franzosa88Cont,FranMisc98}. In essence, the connection and transition matrices appear as a topological approach in order to respond to whether there are connections between Morse sets in a Morse decomposition.

Suppose we have a semigroup $\{S(t):t\geq 0\}$ and consider $K$, an isolated invariant set of $X$.

\begin{definition}
	Given $n \in \mathbb{N}$, a family 
	\[
	M(K)=\{M(1), M(2),\dots, M(n)\}
	\]
	is called a \emph{Morse decomposition} of $K$ if:
	
	\begin{itemize}
		\item[i)] $M(j)$ is an isolated set, for all $j=1,\dots, n$;
		
		\item[ii)] For any global solution $\xi :\mathbb{R}\rightarrow X$ with $\xi (%
		\mathbb{R})\subset K$ we have that either $\xi (\mathbb{R})\subset M(j)$,
		for some $j$, or $\xi \left( \text{\textperiodcentered }\right) $ satisfies 
		\[
		M(k)\overset{t\rightarrow -\infty }{\longleftarrow }\xi (t)\overset{%
			t\rightarrow +\infty }{\longrightarrow }M(j),
		\]%
		for $k>j$.
	\end{itemize}
\end{definition}

\begin{remark}
Instead of indexing the Morse sets in an ``interval'' of $\mathbb{N}$, we can index it in any set that admits a partial order. Consider the ordered pair $(P, <)$ such that $P$ is a set and $<$ is a partial order defined in $P$. $I\subset P$ is an \emph{ interval } if $a,b \in I$ and $c \in P$ such that $a<c<b$ implies that $c \in I$. Given two partial orders $<'$ and $<$ in $P$, we say that $<'$ is an extension of $<$ if, for $a,b \in P$, $a<b$ implies $a<'b$. 
\end{remark}

Consider that we find a Morse decomposition $\{M(\theta):\theta \in P\}$ for some isolated set $K$ under the action of $\{S(t):t\geq 0\}$. The partial order $<$ defined in $P$ gives rise to what we call \emph{admissible order} in $K$: we say that $\theta < \theta'$ if $M(\theta)$ appears before $M(\theta')$ in the Morse decomposition.

For $\theta,\theta' \in P$, we say $\theta <_{F} \theta'$, if there are $\theta_j \in P$ and global solutions $\xi_j:\mathbb{R}\to X$ such that
$$M(\theta_{j+1})\stackrel{t\rightarrow -\infty}{ \longleftarrow } \xi_j(t)\stackrel{t\rightarrow +\infty}{ \longrightarrow }
M(\theta_j),\quad 0\leq j\leq n+1,$$
$\theta_0=\theta$ and $\theta_{n+1}=\theta'$, for some  $n \in \mathbb{N}$. 

Now, $<_{F}$ is called the \emph{flow order}. Also, any admissible order is an extension of the flow order.

We will present the concept of connection matrix applied to our context. For more details on the general theory of connection matrices and the fact that we can specify it as we will do, see \cite{Franzosa89CM}.

Consider a Morse decomposition $\mathcal{M}=\{M(\pi ):\pi \in P\}$ related
to a partial order $<$. Denote by $\{H^{\ast }(\pi )\}_{\pi \in P}$ a
collection of graded modules, where $H^{\ast }(\pi )$ represents the
homology chain of the $\mathbb{Z}$-modules associated to $M(\pi )$, $\pi \in
P$. Recall that the connection matrix is a linear map defined on the graded
modules generated by the sum of the elements in $\{H^{\ast }(\pi )\}_{\pi
	\in P}$ such that the homology index braid generated by $\Delta $ is
isomorphic to the homology index braid of the Morse decomposition (see \cite{Franzosa89CM} for more details). Hence, we define the linear map 
\[
\Delta :\bigoplus_{\pi \in P}H^{\ast }(\pi )\rightarrow \bigoplus_{\pi \in
	I}H^{\ast }(\pi ),
\]%
which can be written as a matrix operator $\Delta =\left( \Delta _{\pi ,\pi
	^{\prime }}\right) _{\pi ,\pi ^{\prime }\in P}$. In \cite{Franzosa89CM} it is
proved that such connection matrix always exists. In addition, it satisfies
the following properties:

\begin{itemize}
	\item[i)] $\Delta $ is an upper triangular matrix, that is, $\Delta _{\pi
		,\pi ^{\prime }}=0$ if $\pi ^{\prime }<\pi $.
	
	\item[ii)] $\Delta $ is a boundary map, that is, $\Delta ^{2}=0$ and $\Delta 
	$ has degree $-1$.
	
	\item[iii)] If $<$ is the flow order $<_{F}$, $\Delta _{\pi ,\pi ^{\prime
	}}\neq 0$ and $\{\pi ,\pi ^{\prime }\}$ is an interval, there is a global
	solution $\xi :\mathbb{R}\rightarrow X$ satisfying 
	\[
	M(\pi ^{\prime })\overset{t\rightarrow -\infty }{\longleftarrow }\xi (t)%
	\overset{t\rightarrow +\infty }{\longrightarrow }M(\pi ).
	\]
\end{itemize}

If we denote by $\Delta _{I}$ the restriction of the map $\Delta $ to any
interval $I$, then all the above properties are also satisfied.

In the sequel, we will always refer to a connection matrix related to the
flow order.

In \cite{Mischaikow95}, the author shows that problems satisfying some conditions have an attractor with the same structure of the Chafee-Infante problem. The conditions for such problems are the following:

	\textbf{(A1)} Consider a sequence $\{\lambda_n\}_{n\in \mathbb{N}} \in \mathbb{R}^+$, with $\lambda_n<\lambda_{n+1}$, for all $n \in \mathbb{N}$. Suppose that we can define a continuous parameterized family of semiflows $\varphi_\lambda :\mathbb{R}^+\times X  \to X $, for $\lambda \in \mathbb{R}^+$.
		
For each $\lambda >0$, we also assume that $\varphi_\lambda$ has a global attractor $\mathcal{A}_\lambda$ and the map
	$$
	\varphi_\lambda: \mathbb{R}\times \mathcal{A}_\lambda\to \mathcal{A}_\lambda
	$$
defines a flow.
\smallskip

\textbf{(A2)} For each $\lambda \in (\lambda_n,\lambda_{n+1})$, the attractor $\mathcal{A}_\lambda$ admits a Morse decomposition 
$$
M_\lambda(\mathcal{A}_\lambda)=\left\{M_\lambda(j^\star): j \in \{0,\dots,N-1\}, \star \in\{+,-\}\right\}\cup \left\{M_\lambda(n)\right\}.
$$

$M_\lambda(n)$ represents the zero equilibrium for $\lambda \in (\lambda_n,\lambda_{n+1})$. Moreover,
	$$\begin{aligned}
		&j^\pm< k^\pm \mbox{ for } j,k \in \{0,\dots,N-1\}\iff  j<k \mbox{ in }\mathbb{N},\\
		&j^\pm < N, \mbox{ for all } j \in \{0,\dots,N-1\}.
	\end{aligned}
	$$ 

is an admissible order.

\smallskip

\textbf{(A3)} We assume that we have the following homology index, for $\lambda \in (\lambda_n, \lambda_{n+1})$:
$$
H^k(M_\lambda(j^\star))\simeq \begin{cases}
	\mathbb{Z}, \mbox{ if } k=j, \\
	0, \mbox{ otherwise},
\end{cases}
\quad\mbox{and}\quad
H^k(M_\lambda(n))\simeq \begin{cases}
	\mathbb{Z}, \mbox{ if } k=n, \\
	0, \mbox{ otherwise,}
\end{cases}
$$
for $j \in \{0,\dots, n-1\}$ and $\star \in \{+,-\}$.
\smallskip

\textbf{(A4)} For $\lambda \in (\lambda_n, \lambda_{n+1})$, the connection matrix is given by 
$$
\Delta_\lambda = \begin{bmatrix}
	  0 &D_1^\lambda& 0   &	  \dots   &         0 \\
	    &0	& D_2^\lambda & \ddots      &  \vdots \\
 \vdots &   & \ddots    &\ddots& 0 \\
       &   && 0    & D_{n}^\lambda    \\
	   0&  &    &\dots & 0 \\	 
\end{bmatrix}
$$
where the submatrices 
$$ D_j^\lambda: H^{j}(M_\lambda(j^-))\oplus  H^j(M_\lambda(j^+))\to  H^{j-1}(M_\lambda((j-1)^-))\oplus  H^{j-1}(M_\lambda((j-1)^+))
$$
can be written as $D_j^\lambda=\begin{bmatrix}
	1&1\\
	-1&-1
\end{bmatrix}$
and
$$ D_{n}^\lambda: H^n(M_\lambda(n)) \to H^{n-1}(M_\lambda((n-1)^-))\oplus  H^{n-1}(M_\lambda((n-1)^+)) $$
can be written as $D_{n}^\lambda=\begin{bmatrix}
	1\\
	-1
\end{bmatrix}$.

Our goal in this section is to show that all the properties \textbf{(A1)}-\textbf{(A4)} are valid for \eqref{eq:nonlocal}. Consequently, the attractor of \eqref{eq:nonlocal} can be well-described and it will have the same structure of the attractor of the Chafee-Infante equation. 

For convenience, we will study the semilinear problem \eqref{eq:nl_semilinear}, since its attractor has the same structure as the attractor for \eqref{eq:nonlocal}.

Observe that the condition \textbf{(A1)} is satisfied for \eqref{eq:nl_semilinear} by the proof of Proposition \ref{Prop:inject.A}, in the previous section.

Consider the sequence $\{\lambda_n\}_{n \in \mathbb{N}}$ given by $\lambda_n =a(0)n^2$, $n \in \mathbb{N}$. Suppose that $a(0)n^2< \lambda < a(0)(n+1)^2$, $n \in \mathbb{N}$. We define the set
\begin{equation}\label{eq:MorseDecomp}
\mathcal{M}=\{M(j^\star): j=0,\dots, N;\ \star\in \{-,+\}\ \}\cup \{M(n)\},
\end{equation}
where $M(n)=\{0\}$ and $M(j^+)=\{\phi_{j+1}^+\}$ and $M(j^-)=\{\phi_{j+1}^-\}$, for $j=0,\dots, n-1$. Although the Morse decomposition \eqref{eq:MorseDecomp} depends on the parameter $\lambda$, we made the choice of simplifying the notation and do not present explicitly this dependency.

\begin{lemma}
	\label{NoConnections}There cannot be an heteroclinic connection between $%
	\phi _{j}^{+}$ and $\phi _{j}^{-}$ for $j=1,...,n$.
\end{lemma}

\begin{proof}
	When $j$ is even, the property $\phi _{j}^{-}\left( x\right) =\phi
	_{j}^{+}\left( \pi -x\right) $, for $x \in [0,\pi]$, implies that $E(\phi _{j}^{-})
	=E(\phi _{j}^{+}) $, where $E$ is the Lyapunov function (\ref{energy}) (see \cite[Lemma 7]{CCM-RV21C-Inf} for the
	details), so by the properties of Lyapunov functions the heteroclinic
	connection is not possible.
	
	When $j$ is odd, we make use of the lap-number property given in \cite[%
	Theorem 6]{art-henry-MS}. For a global solution $u\left( \text{%
		\textperiodcentered }\right) $ let%
	\begin{eqnarray*}
		Q^{+}\left( t\right)  &=&\{x\in \left( 0,1\right) :u\left( t,x\right) >0\},
		\\
		Q^{-}\left( t\right)  &=&\{x\in \left( 0,1\right) :u\left( t,x\right) <0\}.
	\end{eqnarray*}%
	In the proof of Theorem 6 in \cite{art-henry-MS} it is shown that if $t_{1}>t_{0}$,
	then there is an injective map for the connected components of $Q^{+}\left(
	t_{1}\right) $ ($Q^{-}\left( t_{1}\right) $) to connected components of $%
	Q^{+}\left( t_{0}\right) $ ($Q^{-}\left( t_{0}\right) $). Let, for example, $%
	u\left( \text{\textperiodcentered }\right) $ be a global solution such that $%
	u\left( t\right) \underset{t\rightarrow -\infty }{\rightarrow }\phi
	_{j}^{-},\ u\left( t\right) \underset{t\rightarrow +\infty }{\rightarrow }%
	\phi _{j}^{+}$. Since we have convergence in $C^{1}([0,\pi ])$, there are $%
	t_{0}<t_{1}$ such that the number of components of $Q^{+}\left( t_{0}\right) 
	$ is equal to $\left( j-1\right) /2$ and the number of components of $%
	Q^{+}\left( t_{1}\right) $ is equal to $\left( j+1\right) /2$. This
	contradicts the existence of the above injective map. Thus, such
	heteroclinic connection is impossible. A similar argument (but using $%
	Q^{-}\left( t\right) $) is valid for the connection from $\phi _{j}^{+}$ to $%
	\phi _{j}^{-}.$
\end{proof}

\begin{remark}
	The proof of Lemma \ref{NoConnections} is also valid if we use the
	conditions on $f$ and $a$ given in \cite{CCM-RV21C-Inf}, because Theorem 6 from 
	\cite{art-henry-MS} is valid for non-classical solutions as well (see \cite[%
	Theorem A3]{CCM-RV21C-Inf}). In particular, the function $f$ is not necessarily odd. This completes the proof of Lemma 7 from \cite%
	{CCM-RV21C-Inf}, which was given only for $j$ even.
\end{remark}

\begin{lemma}\label{lemma:AdmOrder} For any $a(0)N^2< \lambda < a(0)(N+1)^2$, $\mathcal{M}$ is a Morse decomposition. Moreover, $$\begin{aligned}
			&j^\pm< k^\pm \mbox{ for } j,k \in \{0,\dots,N-1\}\iff  j<k \mbox{ in }\mathbb{N},\\
			&j^\pm < N, \mbox{ for all } j \in \{0,\dots,N-1\}.
		\end{aligned}
		$$ 
	is an admissible order.
\end{lemma}
\begin{proof}
By Lemma \ref{NoConnections}, there cannot exist connections between equilibria with the same number of zeros.
	
	Now, suppose that we have $M(k^{\star})$ and $M(j^{\blacktriangle})$, for $k,j\in \{0,\dots, (N-1), N\}$, $\star, \blacktriangle \in \{\{\emptyset\},+,-\}$ and a global solution $\xi:\mathbb{R}\to X$ satisfying
	$$
	M(k^{\star})\stackrel{t\rightarrow -\infty}{ \longleftarrow } \xi(t) \stackrel{t\rightarrow +\infty}{ \longrightarrow }M(j^{\blacktriangle}).
	$$

The solution $\xi$ is called a connection from $M(k^\star)$ to $M(j^\blacktriangle)$. We denote by $C(M(k^\star),M(j^\blacktriangle))$ the set of all connections from $M(k^\star)$ to $M(j^\blacktriangle)$.
	
By the lap-number property, \cite[Theorem 6]{art-henry-MS}, we have $k\geq j$. The previous arguments exclude the case $k=j$. It follows that $k> j$, as desired.
	
	Therefore, $\mathcal{M}$ is a Morse decomposition and the described order is admissible.     
\end{proof}

Thus \textbf{(A2)} is satisfied. Now, \textbf{(A3)} is assured by Theorem \ref{theo:Conley_index}.

Before we continue, we will present results and definitions from \cite{Franzosa88Cont}, in our context. For a more details and a more general approach, we recommend the cited reference.

Observe that, for each $\lambda \in (a(0)n^2, a(0)(n+1)^2)$, $n\in \mathbb{N}$, and $\tau \in [0,1]$, problem \eqref{eq:nonlocal.tau} admits a global attractor, which we will denote by $\mathcal{A}^\tau$. Again, we omit the dependence of $\lambda$ in the notation. Moreover, applying Proposition \ref{Prop:inject.A} to $a_\tau(\cdot)$ instead of $a(\cdot)$, we also obtain that the semigroup $\{S_\tau(t):t \geq 0\}$ defines a flow
$$
\varphi_\tau :\mathbb{R}\times \mathcal{A}^\tau \to \mathcal{A}^\tau .
$$

\begin{lemma}\label{lemma:UpperSemicA}
	The family $\{\mathcal{A}^\tau: \tau \in [0,1]\}$ is upper semicontinuous.
\end{lemma}
\begin{proof}
	It is not difficult to see that $\bigcup_{\tau \in [0,1]} \mathcal{A}^\tau$ is a bounded set in $X^\gamma$, for some $\gamma \in (\frac{1}{2},1]$. This follows by the same reasoning applied to obtain \eqref{eq:bounded_frac_space} together with the non-degeneracy of $a(\cdot)$ and the dissipativy condition \eqref{eq:dissip_f}.
	
	Hence,
	\begin{equation*}
		\overline{\bigcup_{\tau \in [0,1]} \mathcal{A}^\tau} \mbox{ is compact in } H^1_0(0,\pi).
	\end{equation*}
	
	Now, the family of semigroups $\{S_\tau(t): t \geq 0\}$, $\tau \in [0,1]$, is continuous in the sense of Lemma \ref{lemma:Cont_family}.
	
	Finally, we apply Theorem 3.6 in \cite{CLRLibro} and we obtain the upper semicontinuity of the attractors.
\end{proof}

We can define 
$$\mathcal{I}=\{I \subset X: I \mbox{ is an isolated invariant  set of }S_\tau(\cdot), \mbox{ for some }\tau \in [0,1]\}.
$$
 
Given a compact $N\subset \mathcal{A}_{[0,1]}:=\bigcup_{\tau \in [0,1]}\mathcal{A}^\tau$, we may associate the maps
$$
\Lambda(N)=\{\tau \in [0,1]: N \mbox{ is an isolating neighborhood in } \mathcal{A}^\tau\}
$$
and
$\sigma_N: \Lambda(N)\to \mathcal{I}$ given by $\sigma_N(\tau)=S_\tau$, where $S_\tau$ is the largest invariant set of $S_\tau(\cdot)$ in $N$.

Define the sets
$$
\mathcal{M}_P\!\!=\!\!\!\!\bigcup_{\tau \in [0,1]}\!\!\!\{\left(\mathcal{M}^\tau\!,\mathcal{A}^\tau\!\right)\!:\ \!\mathcal{M}^\tau\!=\!\{M^\tau\!(\pi):\pi \in P\} \mbox{ is a Morse decomposition of } \mathcal{A}^\tau\},
$$
$$\begin{aligned}
\mathcal{M}_<=\{(\mathcal{M},\mathcal{A})\in \mathcal{M}_P: &  \mbox{ if } \mathcal{M}=\{M(\pi):\pi \in P\} \mbox{ and, for }\pi, \pi' \in P, \mbox{ there is }\\ 
&\ \gamma \in C(M(\pi),M(\pi')),\mbox{ then } \pi'<\pi \},\end{aligned}
$$ 
where $<$ represents any partial order in $P$.

It is known, \cite[Proposition 4.4]{Franzosa88Cont}, that $\mathcal{I}$ is a topological space with the topology generated by the following basis
$$
\mathcal{B}=
\bigcup_{\stackrel{N \subset \mathcal{A}_{[0,1]}}{N \mbox{\scriptsize compact}}}\{\sigma_N(U) \subset X:\ U \mbox{ is an open set with } U\subset \Lambda(N) \}.
$$ 

Thus, $\mathcal{M}_P$ and $\mathcal{M}_<$ are topological spaces with the topology induced as subspaces of $\left(\prod_{\pi \in P}\mathcal{I}\right)\times \mathcal{I}$.

	\begin{definition} 
	The collection $\mathcal{M}=\{M(\pi)\}_{\pi \in P}$  is called a $<-$ordered Morse decomposition of $\mathcal{A}$ if $(\mathcal{M},\mathcal{A}) \in \mathcal{M}_P$ and if $\gamma \in \mathcal{A}\setminus \cup_{\pi\in P}M(\pi)$, then there is $\pi <\pi'$ with $\gamma \in C(M(\pi'),M(\pi))$.
\end{definition}
		
\begin{definition}
	Let $\mathcal{M}^\tau=\{M^\tau(\pi)\}_{\pi \in P}$ and $\mathcal{M}^{\tilde{\tau}}=\{M^{\tilde{\tau}}(\pi)\}_{\pi \in P}$ be Morse decompositions of $\mathcal{A}^\tau$ and $\mathcal{A}^{\tilde{\tau}}$, respectively.
	
	We say that $M^\tau$ and $M^{\tilde{\tau}}$ are related by continuation or are continuations of each other if there is a path $c$ in $\mathcal{M}_P$ from $\prod_{\pi \in P} M^\tau(\pi)\times \mathcal{A}^\tau$ to $\prod_{\pi \in P} M^{\tilde{\tau}}(\pi)\times \mathcal{A}^{\tilde{\tau}}$. If, furthermore, $M^\tau$ and $M^{\tilde{\tau}}$ are $<-$ordered and the path $c$ is in $\mathcal{M}_<$, then we say that the associated admissible orderings are related by continuation or are continuations of each other.
\end{definition}

\begin{theorem}[Corollary 5.6, \cite{Franzosa88Cont}] \label{theo:F88_corol}
	If the flow ordering of $\mathcal{M}$ is related by continuation to an admissible ordering of $\tilde{\mathcal{M}}$ then the set of connection matrices of $\tilde{\mathcal{M}}$ is a subset of the set of connection matrices of $\mathcal{M}$.
\end{theorem}

Fix $\lambda\in \left( a(0)N^{2},a(0)\left( N+1\right) ^{2}\right)$ and let $\{S_\tau(t):t\geq 0\}$, $\tau \in [0,1]$, be the family of semigroups from Lemma \ref{lemma:cont.equilibria}. Denote $$\mathcal{E}^\tau=\{0\}\cup \{\phi_{j,\tau}^+,\phi_{j,\tau}^-:j=1,\dots, N\}.$$ 


Then $\mathcal{M}^\tau =\{M^\tau(0^+),M^\tau(0^-)\dots, M^\tau((N-1)^+),M^\tau((N-1)^-), M^\tau(N)\}$ is a Morse decomposition for $M^\tau(N)=\{0\}$, $M^\tau(j^+)=\{\phi_{j+1,\tau}^+\}$ and $M^\tau(j^-)=\{\phi_{j+1,\tau}^-\}$ for $j=1,\dots,N-1$.

We have the following result:
\begin{lemma}\label{lemma:path_c}
	The map 
	$$\begin{aligned}
		c: [0,1]&\to \mathcal{M}_P \\
		\tau &\mapsto c(\tau)=(\mathcal{M}^\tau, \mathcal{A}^\tau)
	\end{aligned}.
	$$
is a path in $\mathcal{M}_P$.
\end{lemma} 
\begin{proof}
Consider $\tau_0 \in [0,1]$ and let $\mathcal{V}=\left(\prod_{k \in \{0^\pm,\dots,(N-1)^\pm, N\}} V_k\right)\times V_{\tau_0}$ be an open set in $\mathcal{M}_P$ that contains $c(\tau_0)$. Hence, we have
$$
M^{\tau_0}(k) \subset V_k, \mbox{ for all } k=0^\pm,\dots,(N-1)^\pm, N, \mbox{ and } \mathcal{A}^{\tau_0} \subset V_{\tau_0}.
$$
Since $\mathcal{B}$ is a basis for $\mathcal{M}_P$, we may assume, w.l.g, that we can find compact sets $N_k, N_{\tau_0} \subset X$ and open sets $U_k\subset \Lambda(N_k)$, $U_{\tau_0}\subset \Lambda(N_{\tau_0})$ such that 
$$
V_k=\sigma_{N_k}(U_k) \mbox{ for all } k=0^\pm,\dots,(N-1)^\pm, N, \mbox{ and } V_{\tau_0}=\sigma_{N_{\tau_0}}(U_{\tau_0}) .
$$

Take $U=U_{\tau_0}\bigcap_{k=0\pm,\dots, (N-1)^\pm, N} U_k$. Then, for any $\tau \in U$, we have that
$$
M^\tau(k) \mbox{ is the largest invariant set in }N_k, \ k=0^\pm,\dots,(N-1)^\pm, N
$$
$$
\mathcal{A}^\tau \mbox{ is the largest invariant set in }N_{\tau_0}.
$$

Hence $c(\tau)=(\mathcal{M}^\tau, \mathcal{A}^\tau)\in \mathcal{V}$, for all $\tau \in U$.

It is clear that $U$ is a not empty and open set in $[0,1]$, by the continuity of the Morse decomposition and the upper semicontinuity of the family of attractors.

Therefore, $c$ is a path in $\mathcal{M}_P$.
\end{proof} 
 
\begin{theorem}
For any $\lambda \in\left( a(0)N^{2},a(0)\left( N+1\right) ^{2}\right)$, the connection matrix associated with the Morse decomposition \eqref{eq:MorseDecomp} is as in property \textbf{(A4)}. 
\end{theorem}
\begin{proof}
Observe that, for $\tau=0$, we have the problem
\begin{equation}\label{eq:tau_equal0}
	\left\{	\begin{aligned}
		&u_t = \bar{a} u_{xx}+\lambda f(u),\ x \in (0,\pi), \ t >0,\\
		& u(t,0)=u(t,\pi)=0, \ t \geq 0,\\
		& u(0,\cdot)=u_0 \in H^1_0(0,\pi),
	\end{aligned}\right.
\end{equation}
for some constant $\bar{a}>0$ (which is given by $a(\|(\phi_N^\pm)_x\|^2)$, under the same notation of Lemma \ref{lemma:cont.equilibria}).

Since \eqref{eq:tau_equal0} is a Chafee-Infante equation, by \cite{art-henry-MS}, the flow ordering $<_F$ of $\varphi_0$
is given by  
\begin{equation}\label{eq:orderF}
\begin{aligned}
	&j^\pm<_F k^\pm \mbox{ for } j,k \in \{0,\dots,N-1\}\iff  j<k,\\
	&j^\pm <_F N, \mbox{ for all } j \in \{0,\dots,N-1\}.
\end{aligned}
\end{equation}

By \cite{Mischaikow95}, the only connection matrix when $\tau=0$ is given by
\begin{equation}\label{eq:ConMatrix}
\Delta =%
\begin{bmatrix}
	0 & D_{1} & 0 & \dots  & 0 \\ 
	& 0 & D_{2} & \ddots  & \vdots  \\ 
	\vdots  &  & \ddots  & \ddots  & 0 \\ 
	&  &  & 0 & D_{n} \\ 
	0 &  &  & \dots  & 0%
\end{bmatrix}%
\end{equation}
as in \textbf{(A4)}.

Remember that $a(0)N^2< \lambda < a(0)(N+1)^2$ implies $a_\tau(0)N^2< \lambda < a_\tau(0)(N+1)^2$, for all $\tau \in [0,1]$. 

Applying Lemma \ref{lemma:AdmOrder} to \eqref{eq:nonlocal.tau}, we have, for any $\tau \in [0,1]$, the partial order $<_F$, given in \eqref{eq:orderF},
is an admissible order for $\mathcal{M}^\tau$, $\tau \in [0,1]$.

Hence, $c([0,1])\subset \mathcal{M}_{<_F}$, for the function $c$ presented in Lemma \ref{lemma:path_c}. Since $\mathcal{M}_{<_F}$ is open in $\mathcal{M}_P$ (see \cite[Proposition 4.14]{Franzosa88Cont}), it follows that $c$ is also a path in $\mathcal{M}_{<_F}$.

By Theorem \ref{theo:F88_corol}, for each $\tau \in [0,1]$, the set of connection matrices related to $\mathcal{M}^\tau$ is unitary, whose element is given in \eqref{eq:ConMatrix}.

Now, we just need to observe that for $\tau =1$ problem \eqref{eq:nonlocal.tau} represents \eqref{eq:nonlocal}.

Therefore, the condition \textbf{(A4)} is satisfied for \eqref{eq:nonlocal}, as desired.
\end{proof}

In order to obtain the structure of the global attractor for the problem \eqref{eq:nonlocal}, we will present Theorem 1.2, \cite{Mischaikow95}. In order to do that, we need to define an auxiliary problem. Fix $n \in \mathbb{N}$. Denote by $\|\cdot\|_{\mathbb{R}^n}$ and $\left<\cdot,\cdot\right>$ the norm and inner product, respectively, of the Euclidean space $\mathbb{R}^n$. Consider the subsets 
$$
D^n = \{x=(x_1,\dots, x_n)\in \mathbb{R}^n: \|x\|_{\mathbb{R}^n}\leq 1\},
$$
$$
S^{n-1}= \{x=(x_1,\dots, x_n)\in \mathbb{R}^n: \|x\|_{\mathbb{R}^n}= 1\}
$$
and, for $j=1,\dots,n$, define $e_j^\pm=(\delta_{1j}^\pm,\dots, \delta_{nj}^\pm)$ with $\delta_{jj}=\pm 1$ and $\delta_{kj}^\pm=0$, if $k\neq j$.

 
Consider the following problem
\begin{equation}\label{eq:ODE}
\left\{\begin{aligned}
&\dot{\theta} = Q\theta- \left<Q\theta,\theta\right>\theta, \ \theta \in S^{n-1},\\
&\dot{r}=r(1-r), \ r \in [0,1],
\end{aligned}
\right.
\end{equation}
for 
$$
Q=\begin{bmatrix}
1 & 0 & \dots & 0 \\
0 & \tfrac{1}{2}& &\\
\vdots & & \ddots & \\
0 & & & \tfrac{1}{n}
\end{bmatrix} .
$$

It can be shown that problem \eqref{eq:ODE} defines a flow 
$$
\psi^n :\mathbb{R}\times D^n\to D^n.
$$

It has been proved in \cite[Theorem 1.1]{Mischaikow95} that the dynamics inside the attractor of \eqref{eq:C-I} is conjugated to the dynamics defined by $\psi^n$ on $D^n$, for $\lambda \in (n^2, (n+1)^2)$, $n \in \mathbb{N}$. Moreover, we have the following result

\begin{theorem}[Theorem 1.2, \cite{Mischaikow95}]\label{theo:MischaikowT1.2}
Assuming \emph{\textbf{(A1)-(A4)}} and $\lambda \in (\lambda_n,\lambda_{n+1})$, there exists a flow $\tilde{\varphi}_\lambda$ given by a time-reparameterization of $\varphi_\lambda$ and a continuous surjective map $g_\lambda:\mathcal{A}_\lambda\to D^n$, satisfying $M(j^\pm)=g_\lambda^{-1}(\{e_{j+1}^\pm\})$, for $0\leq j\leq n-1$, and $M(n)=g_\lambda^{-1}(\{0\})$, such that the following diagram commutes:
$$
\xymatrix{
\mathbb{R}\times \mathcal{A}_\lambda \ar[d]_{\tilde{\varphi}_\lambda} \ar[r]^{I_\mathbb{R}\times g_\lambda} & \mathbb{R}\times D^n \ar[d]^{\psi^n} \\
\mathcal{A}_\lambda \ar[r]^g & D^n 
}
$$
where $I_\mathbb{R}$ represents the identity in $\mathbb{R}$.
\end{theorem}

Thus, for what it was shown in the section, together with Theorem \ref{theo:MischaikowT1.2}, we conclude the following
\begin{theorem}
The global attractor $\mathcal{A}_\lambda$ of \eqref{eq:nonlocal}, for $a(0)N^2 <\lambda <a(0)(N+1)^2$, has the same structure of the global attractor $\tilde{\mathcal{A}}_{\frac{\lambda}{a\!(0)}}$ of \eqref{eq:C-I}.
\end{theorem}
\begin{proof}
Since problem \eqref{eq:nonlocal} satisfies \textbf{(A1)-(A4)}, by Theorem \ref{theo:MischaikowT1.2}, the attractor $\mathcal{A}_\lambda$ has the same structure of the attractor for \eqref{eq:ODE} for $n=N$. By \cite[Theorem 1.1]{Mischaikow95}, this is the same structure of the attractor of \eqref{eq:C-I}, for $\lambda \in (N^2,(N+1)^2)$ . 

\end{proof}
Thus, for any global bounded solution $\xi:\mathbb{R}\to H^1_0(0,\pi)$, we find $\phi_j,\phi_k\in \mathcal{E}_\lambda$ such that
\begin{equation*}
	\phi_j \stackrel{t\rightarrow -\infty}{\longleftarrow} \xi(t)  \stackrel{t\rightarrow +\infty}{\longrightarrow} \phi_k.
\end{equation*}

Also, either $\phi_j=\{0\}$ or $\phi_j$ has at least one more zero in $[0,\pi]$ than $\phi_k$, if $\xi(\cdot)$ is not an equilibrium.

If $\phi_j,\phi_k\in \mathcal{E}_\lambda$ and $\phi_j$ has at least one more zero in $[0,\pi]$ than $\phi_k$ or $\phi_j=0$, then we find a global bounded solution $\eta:\mathbb{R}\to H^1_0(0,\pi)$  satisfying
\begin{equation*}
	\phi_j \stackrel{t\rightarrow -\infty}{\longleftarrow} \eta(t)  \stackrel{t\rightarrow +\infty}{\longrightarrow} \phi_k.
\end{equation*}

In \cite{Fiedler_Rocha_Het_orb}, the authors pictured the connections in $\tilde{\mathcal{A}}_\lambda$ as this nice diagram, where the oriented paths determined by the arrows imply the existence of connections between the initial and the final equilibria of the path:

$$
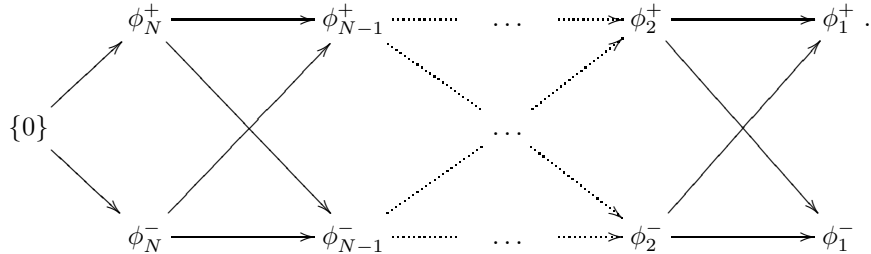
\xymatrix{ & \phi_N^+ \ar[rr]\ar[rrdd]& &\phi_{N-1}^+\ar@{.}[rd] \ar@{.}[r] & \phantom{I_1}\dots\phantom{I_1} \ar@{.>}[r] & \phi_2^+ \ar[rr] \ar[rrdd] & & \phi_1^+
	\\
	\{0\} \ar[ru]\ar[rd] & & &  &\phantom{I_1} \dots \phantom{I_1}  \ar@{.>}[rd]\ar@{.>}[ru]&&&
	\\
	& \phi_N^- \ar[rruu]\ar[rr]& & \phi_{N-1}^- \ar@{.}[ru] \ar@{.}[r] &\phantom{I_1}\dots \phantom{I_1}\ar@{.>}[r]& \phi_2^- \ar[rr] \ar[rruu]& & \phi_1^-
}.
$$

\begin{remark}

In the construction above, we will construct a function $a$ which admits at least three positive equilibria, and not all stable. In order to make this more clear, we will introduce some results that can be found in \cite{CCM-RV21C-Inf}.

\begin{proposition}
	Assume that $a$ is non-decreasing. 
	Suppose that $\lambda>a(0)$ and, for each $d\in \mathbb{R}$, consider
	\begin{equation}\label{eq:C-Inf_var.d}
		\left\{\begin{aligned}
			& u_t=a(d)u_{xx}+\lambda f(u),\ t>0,\ x \in (0,\pi), \\
			& u(t,0)=u(t,\pi)=0, \ t\geq0, \\
			& u(0,\cdot)=u_0(\cdot) \in H^1_0(0,\pi), 
		\end{aligned}\right. 
	\end{equation}
	and $f \in C^2(\mathbb{R})$, satisfying \eqref{eq:prop_f} and $ \limsup_{|s|\rightarrow +\infty}\frac{f(s)}{s}<0$.
	
	Then, for each $d \in \mathbb{R}^+$, as long as $a(d)<\lambda$, we find a positive equilibrium $\phi_d$ of \eqref{eq:C-Inf_var.d}. Also, the function 
	$\mathbb{R}^+ \ni d\mapsto \|\phi_d'\|^2$ is non-decreasing.
\end{proposition}

In what follows, we want to make a choice of function $c: \mathbb{R}\to\mathbb{R}$ in order to construct at least $3$ positive equilibria of
\begin{equation}\label{eq:nleq}
	\left\{\begin{aligned}
		&u_t=c(\Vert u_x\Vert^2)u_{xx}+\lambda f(u), \ t>0, x \in (0,\pi),\\ 
		& u(t,0)=u(t,\pi)=0, \ t\geq0 \\
		& u(0,x)=u_0(\cdot) \in H^1_0(0,\pi) 
	\end{aligned}\right. 
\end{equation}

Initially, consider a non-decreasing function $a:\mathbb{R}^+\to (0,+\infty)$. In this case, for every $d>0$, as long as $\lambda>a(d)$, we find $\phi_d$ a positive equilibrium of \eqref{eq:C-Inf_var.d}.

Observe that if $\phi_d$ is also an equilibrium of \eqref{eq:nonlocal} then $d=\|(\phi_d)_x\|^2$.	 

%
Consider initially that $a$ is \emph{increasing}. Then, if $\lambda > a(0)$ we know that there exists $d^*>0$ such that $\phi_{d^*}$ is an equilibrium for the non-local problem, that is, $\|(\phi_{d^*})_x\|^2=d^*$. 

Denote $d_0=g(0)$ and consider $\delta_0>0$ sufficiently small such that $d^*+4\delta_0 <d_0$. Observe that we can find $\bar{d} \in (0,d^*)$ with $\|\phi_{\bar{d}}\|^2=d^*+3\delta_0$.

For $\delta \in (0,\delta_0)$, define $c_\delta:\mathbb{R}^+ \to \mathbb{R}^+$ by the following: 
\begin{itemize}
	\item[i)] $c_\delta(t)=a(t)$ if $t \in [0, d^*+\delta]$;
	\item[ii)]  In $[d^*+2\delta, d^*+3\delta]$ we define as a segment connecting $c_\delta(d^*+2\delta)=a(d^*)$ and $c(d^*+3\delta)=a(\bar{d}).$ 
	\item[iii)] For $d \in [d^*+4\delta, +\infty)$, define as $c(d)=a(0)$.
	\item[iv)] In the intervals, we did not mention, we define as smooth function in a way that $c \in C^1(\mathbb{R}^+)$.
\end{itemize}


Observe that, for each $\delta \in (0,\delta_0)$, we find $d_\delta \in (d^\star+2\delta,d^\star+3\delta)$ such that $a(\|\phi_{d_\delta}'\|^2)=d_\delta$.

For simplicity, denote $\psi_\delta =\phi_{d_\delta}$. We need to show that we find $\delta>0$ such that we have a positive equilibria such that the linearization
$$
L v=v'' +\lambda\frac{f'(\psi_\delta)}{c(\|\psi'_\delta\|^2)}v-\tfrac{2\lambda^2 c'(\|\psi_\delta'\|^2)}{c(\|\psi'_\delta\|^2)^3}f(\psi_\delta)\int_{0}^{\pi} f(\psi_\delta(s))v(s)ds
$$
has a positive eigenvalue.

The argument to finish that is to observe that, although $\psi_\delta$ depends on our choice of $\delta$, the values of $\psi^\delta$ are bounded.

In fact, we just need to observe that the construction forces $\psi_\delta = \phi_d $, for some $d \in (\bar{d}, d^*)$.

There is a constant $J>0$ such that, for any $\psi =\phi_d$, $d \in (\bar{d}, d^*)$ the linearization 
$$
L v=v'' +\lambda\frac{f'(\psi)}{c(\|\psi'\|^2)}v+\frac{2\lambda^2 J}{c(\|\psi'\|^2)^3}f(\psi)\int_{0}^{\pi} f(\psi(s))v(s)ds
$$
has at least one positive eigenvalue. This comes from Lemma 6.2 in \cite{DavidsonDodds} and the fact that the function $(\bar{d}, d^*)\ni d\mapsto \|(\phi_d)_x\|^2$ is continuous and $f$ is $C^2$, which allows us to choose a $J$ without dependence on $d$, for $(\bar{d}, d^*)$.

Finally, just choose $\delta \in (0,\delta_0)$ small enough such that $c'(u)=-J$, for $u \in (d^*+2\delta, d^*+3\delta)$.
\end{remark}

\section{Final remarks}
In this article, we have proved that, although problem \eqref{eq:nonlocal} is not locally close to a linear problem, not only the equilibria have a ``saddle-point property'' as we already knew, but also the inner structure of the global attractor is very well-understood. We are able to describe all the connections between equilibria.

Just to complement, we have shown that the assumption that $a$ is non-decreasing plays a crucial role in the construction. In fact, in \cite{CCM-RV21C-Inf}, the authors have already shown that the exclusion of such hypothesis may cause the appearance of a lot of equilibria (two or more) in each bifurcation. Now, we have proved, in addition, that some of the positive equilibria can be unstable.

There are other steps we could make. There are variations of \eqref{eq:nonlocal} for which we do not have uniqueness. In that case, the ideas applied here to assure the structure inside the attractor cannot be applied.

\section*{Acknowledgments}

We would like to thank Phillipo Lappicy and professor Alexandre Nolasco de Carvalho for the stimulating discussions and valuable suggestions during the elaboration of the present article.

The first author was partially supported by São Paulo Research Foundation (FAPESP), grants \#2018/00065-9  and \#2020/00104-4, and by the Coordenação de Aperfeiçoamento de Pessoal de Nível Superior – Brasil (CAPES) – PROEX 7547361/D.

The second author has been partially supported by the Spanish Ministry of Science, Innovation and Universities, project PGC2018-096540-B-I00, by the Spanish Ministry of Science and Innovation, project PID2019-108654GB-I00, and by the Junta de Andaluc\'{\i}a and FEDER, project P18-FR-4509.

\bibliography{biblio}
\bibliographystyle{plain}

\end{document}